\documentclass[10pt,english]{amsart}
\usepackage[T1]{fontenc}
\usepackage[latin9]{inputenc}
\usepackage{geometry}
\usepackage{verbatim}
\usepackage{mathrsfs}
\usepackage{amsthm}
\usepackage{amstext}
\usepackage{amssymb}
\usepackage{graphicx}
\usepackage{esint}
\usepackage{color}
\usepackage{srcltx} 

\makeatletter
\numberwithin{equation}{section}
\numberwithin{figure}{section}
\theoremstyle{plain}

  \theoremstyle{definition}
  
  \theoremstyle{plain}
  
  \theoremstyle{plain}
  
  \theoremstyle{plain}
  
  \theoremstyle{plain}
  \newtheorem*{cor*}{\protect\corollaryname}
  \theoremstyle{remark}
  
  \theoremstyle{definition}

\newcommand{\R}{\mathbb{R}}

\newcommand{\M}{\mathcal{M}}

\renewcommand{\H}{\mathcal{H}}

 \newcommand{\tn}{\textnormal}
\newcommand{\po}{\partial} 
 
\newcommand{\var}{\varepsilon}

 \renewcommand{\O}{{\mathbf O}}

 \newcommand{\supp}{\text{\rm supp}\,}

 \renewcommand{\div}{\textnormal{div }}
 \renewcommand{\supp}{\text{\rm supp}\,}

\renewcommand{\O}{\Omega} \newcommand{\Lip}{\text{\rm Lip}}
 
\theoremstyle{plain}
\newtheorem{theorem}{Theorem}[section]
\newtheorem{lemma}[theorem]{Lemma}

\newtheorem{definition}[theorem]{Definition}

\newtheorem{proposition}[theorem]{Proposition}
\newtheorem{corollary}[theorem]{Corollary}
\newtheorem{remark}[theorem]{Remark}

\newcommand{\rnm}{\mathbb{R}^{n}}

\newcommand{\norm}[1]{\left\Vert#1\right\Vert}

\renewcommand{\div}{\text{\sl div\,}}

\newcommand{\FF}{{\boldsymbol F}}

\newcommand{\ve}{{\varepsilon}}

\def\rightangle{\vcenter{\hsize5.5pt
    \hbox to5.5pt{\vrule height7pt\hfill}
    \hrule}}
\def\rtangle{\mathrel{\rightangle}}
\def\intave#1{\int_{#1}\hbox{\llap{$\raise2.3pt\hbox{\vrule
height.9pt width7pt}\phantom{\scriptstyle{#1}}\mkern-2mu$}}}

\def\intav#1{\mathchoice
          {\mathop{\vrule width 9pt height 3 pt depth -2.6pt
                  \kern -9pt \intop}\nolimits_{\kern -6pt#1}}%
          {\mathop{\vrule width 5pt height 3 pt depth -2.6pt
                  \kern -6pt \intop}\nolimits_{#1}}%
          {\mathop{\vrule width 5pt height 3 pt depth -2.6pt
                  \kern -6pt \intop}\nolimits_{#1}}%
          {\mathop{\vrule width 5pt height 3 pt depth -2.6pt
                  \kern -6pt \intop}\nolimits_{#1}}}
\def\intav#1{\vint_{#1}}

\allowdisplaybreaks[1]
\allowdisplaybreaks[1]

\newcommand{\bnu}{\boldsymbol\nu}
\renewcommand{\div}{\text{\sl div\,}}

  \providecommand{\corollaryname}{Corollary}
  \providecommand{\definitionname}{Definition}
  \providecommand{\examplename}{Example}
  \providecommand{\lemmaname}{Lemma}
  \providecommand{\propositionname}{Proposition}
  \providecommand{\remarkname}{Remark}
\providecommand{\theoremname}{Theorem}

\usepackage{babel}
\providecommand{\corollaryname}{Corollary}
  \providecommand{\definitionname}{Definition}
  \providecommand{\examplename}{Example}
  \providecommand{\lemmaname}{Lemma}
  \providecommand{\propositionname}{Proposition}
  \providecommand{\remarkname}{Remark}
\providecommand{\theoremname}{Theorem}

\usepackage{babel}
\providecommand{\corollaryname}{Corollary}
  \providecommand{\definitionname}{Definition}
  \providecommand{\examplename}{Example}
  \providecommand{\lemmaname}{Lemma}
  \providecommand{\propositionname}{Proposition}
  \providecommand{\remarkname}{Remark}
\providecommand{\theoremname}{Theorem}

\makeatother

\usepackage{babel}
  \providecommand{\corollaryname}{Corollary}
  \providecommand{\definitionname}{Definition}
  \providecommand{\examplename}{Example}
  \providecommand{\lemmaname}{Lemma}
  \providecommand{\propositionname}{Proposition}
  \providecommand{\remarkname}{Remark}
\providecommand{\theoremname}{Theorem}

\begin{document}

\title[Measures in the dual of  $BV$]{Characterizations of signed measures in the dual of $BV$\\ and related isometric isomorphisms}

\author{Nguyen Cong Phuc and Monica Torres}

\address{{[}Nguyen Cong Phuc{]}{ Department of Mathematics, Lousiana State University, 303 Lockett Hall, Baton Rouge, LA 70803, USA.
http://www.math.lsu.edu/\~{}pcnguyen} }

\email{pcnguyen@math.lsu.edu}

\address{{[}Monica Torres{]}{ Department of Mathematics\\
 Purdue University, 150 N. University Street, West Lafayette, IN 47907-2067,
USA. http://www.math.purdue.edu/\~{}torres} }

\email{torres@math.purdue.edu}


\begin{abstract}
We characterize all (signed) measures in $BV_{\frac{n}{n-1}}(\R^n)^*$, where $BV_{\frac{n}{n-1}}(\R^n)$ is defined as the space of all functions $u$ in $L^{\frac{n}{n-1}}(\R^n)$ such that $Du$ is a finite vector-valued measure. We also show that $BV_{\frac{n}{n-1}}(\R^n)^*$ and $BV(\R^n)^*$ are isometrically isomorphic, where $BV(\R^n)$ is defined as the space of all functions $u$ in $L^{1}(\R^n)$ such that $Du$ is a finite vector-valued measure. As a consequence of our characterizations,  an old issue raised in Meyers-Ziemer  \cite{MZ} is resolved by constructing a locally integrable function $f$ such that $f$ belongs to $BV(\R^n)^{*}$ but $|f|$ does not. Moreover, we  show that  the  measures in $BV_{\frac{n}{n-1}}(\R^n)^*$ coincide with the measures in $\dot W^{1,1}(\R^n)^*$, the dual of the homogeneous Sobolev space $\dot W^{1,1}(\R^n)$, in the sense of isometric isomorphism. For  a  bounded open set $\O$ with Lipschitz boundary, we characterize the measures in the dual space $BV_0(\O)^*$. One of the goals of this paper is to make precise the definition of $BV_0(\O)$, which is the space of functions of bounded variation with zero trace on the boundary of $\O$. We show that the measures in $BV_0(\O)^*$ coincide with the measures in $W^{1,1}_0(\O)^*$. Finally, the class of finite measures in $BV(\O)^*$ is also characterized. 

 \end{abstract}
\maketitle

\section{Introduction}

It is a challenging problem in geometric measure theory to give a full characterization of the dual of $BV$, the space of functions of bounded variation. Meyers and Ziemer characterized in \cite{MZ} 
the positive measures in $\R^n$ that belong to the dual of $BV(\R^n)$. They defined $BV(\R^n)$ as the space of all functions in $L^1(\R^n)$ whose distributional gradient is a finite vector-measure in $\R^n$ with norm given by
\begin{equation*}
\norm{u}_{BV(\R^n)}= \norm{D u}(\R^n).
\end{equation*}
They showed that the positive measure $\mu$ belongs to $BV(\R^n)^*$ if and only if $\mu$ satisfies the condition
\begin{equation*}
 \mu(B(x,r)) \leq Cr^{n-1}
\end{equation*}
for every open ball $B(x,r) \subset \R^n$ and $C=C(n)$. Besides the classical paper by Meyers and Ziemer, we  refer the interested reader to the paper by De Pauw \cite{Thierrysbv}, where the author analyzes $SBV^*$, the dual of the space of special functions of bounded variation.

In Phuc-Torres \cite{PT} we showed that there is a connection between the problem of characterizing $BV^*$ and the study of the solvability of the equation $\div \FF =T$. Indeed, we showed that the (signed) measure $\mu$ belongs to $BV(\R^n)^*$ if and only if there exists a bounded vector field $\FF \in L^\infty(\R^n,\R^n)$ such that $\div \FF= \mu$. Also, we showed that $\mu$ belongs to $BV(\R^n)^*$ if and only if
\begin{equation}
\label{consmoothsets}
 |\mu (U) | \leq C\, \H^{n-1}(\partial U)
\end{equation}
for any open (or closed) set $U \subset \R^n$ with smooth boundary. The solvability of the equation $\div \FF= T$, in various spaces of functions, has been studied in Bourgain-Brezis \cite{BB}, De Pauw-Pfeffer \cite{PP}, De Pauw-Torres \cite{DPT} and Phuc-Torres \cite{PT} (see also Tadmor \cite{Tadmor}).

In De Pauw-Torres \cite{DPT}, another $BV$-type space was considered, the space $BV_{\frac{n}{n-1}}(\R^n)$, defined as the 
space of all functions $u \in L^{\frac{n}{n-1}}(\R^n)$ such that $D u$, the distributional gradient of $u$, is a finite vector-measure in $\R^n$. A closed subspace of  $BV_{\frac{n}{n-1}}(\R^n)^*$, which is a Banach space denoted as $CH_0$, was characterized in \cite{DPT} and it was proven that $T \in CH_0$ if and only if $T= \div \FF$, for a continuous vector field $\FF  \in C(\R^n,\R^n)$ vanishing at infinity. 
 
In this paper we continue the analysis of $BV(\R^n)^*$ and $BV_{\frac{n}{n-1}}(\R^n)^*$. We show that  $BV(\R^n)^*$ and $BV_{\frac{n}{n-1}}(\R^n)^*$ are isometrically isomorphic (see Corollary \ref{iso-R}). We also show that  the  measures in $BV_{\frac{n}{n-1}}(\R^n)^*$ coincide with the measures in  $\dot W^{1,1}(\R^n)^*$, the dual of the homogeneous Sobolev space $\dot W^{1,1}(\R^n)$ (see Theorem \ref{nuevo}),  in the sense of isometric isomorphism. We remark that the space $\dot W^{1,1}(\R^n)^*$ is denoted as the $G$ space in image processing (see Meyer \cite{Meyer}), and that it plays a key role in modeling the structured component of an image.

It is obvious that if $\mu$ is a locally finite signed Radon measure then $\norm{\mu}\in BV(\R^n)^*$ implies that $\mu\in BV(\R^n)^*$. The converse was unknown to  Meyers and Ziemer as they raised this issue in their classical paper \cite[page 1356]{MZ}. In Section \ref{MZ-ex}, 
we show that the converse does not hold true in general by constructing a locally integrable function $f$ such that $f\in BV(\R^n)^{*}$ but $|f|\not \in BV(\R^n)^{*}$.

 In this paper we also study these characterizations in bounded domains. Given a bounded open set $\O$ with Lipschitz boundary, we consider  the space $BV_0(\Omega)$ defined as the space of functions of bounded variation with zero trace on $\partial \O$. One of the goals of this paper is to make precise the definition of this space (see Theorem \ref{sonlomismo}) . We then characterize all (signed) measures in $\O$ that belong to $BV_0(\Omega)^*$ . We show that a locally finite signed measure $\mu$ belongs to  $BV_0(\Omega)^*$ if and only if \eqref{consmoothsets} holds for any smooth open (or closed) set $U \subset \subset \O$, and if and only if $\mu =\div \FF$ for a vector field $\FF \in L^\infty(\O,\R^n)$  (see Theorem \ref{superestrellabounded}). Moreover, we show that the measures in $BV_0(\O)^*$ coincide with the measures in $W^{1,1}_0(\O)^*$ (see Theorem \ref{nuevo1}), in the sense of isometric isomorphism.

 In the case of $BV(\O)$, the space of functions of bounded variation in  a   bounded open set $\O$ with Lipschitz boundary
(but without the condition of having zero trace on $\partial\O$), we shall restrict 
our attention only to  measures in $BV(\O)^*$ with bounded total variation in $\O$, i.e., finite measures.
This is in a sense natural since any {\it positive} measure that belongs to $BV(\O)^*$ must be finite due to the fact that the function $1\in BV(\O)$.  We show that  a finite measure $\mu$ belongs to $BV(\O)^*$ if and only if \eqref{consmoothsets} holds for every smooth open set $U \subset \subset \R^n$,
where $\mu$ is extended by zero to $\R^n\setminus\O$ (see Theorem \ref{resultforbounded}).

\section{Functions of bounded variation}
In this section we define all the spaces that will be relevant in this paper.
\begin{definition}
Let $\O$ be any open set. The space $\M(\O)$ consists of all finite (signed) Radon measures $\mu$ in $\O$; that is, the total variation of $\mu$, denoted as $\norm{\mu}$, satisfies  $\norm{\mu}(\O) < \infty$. The space $\M_{loc}(\O)$ consists of all locally finite  Radon measures $\mu$ in $\O$; that is, $\norm{\mu}(K) < \infty$ for every compact set $K \subset \O$. 
\end{definition}

Note here that  $\M_{loc}(\O)$ is identified with the dual of the locally convex space
$C_c(\O)$ (the space of continuous real-valued functions with compact support in $\O$) (see \cite{Di}), and thus it is a real vector space. For  $\mu \in \M_{loc}(\O)$,
it is not required that either the positive part or the negative part of $\mu$ has finite total variation in $\O$.

\begin{definition}
Let $\O$ be any open set. The space of functions of bounded variation, denoted as $BV(\O)$, is defined as the space of all functions $u \in L^1(\O)$ such that  the distributional gradient $Du$ is a finite vector-valued measure in $\O$. The space $BV(\O)$ is a Banach space with the norm
\begin{equation}
\label{normaenbv}
\norm{u}_{BV(\O)}= \norm{u}_{L^1(\O)} + \norm{D u} (\O),
\end{equation}
where $\norm{D u} (\O)$ denotes the total variation of the vector-valued measure $Du$ over $\O$. For the case when $\O=\R^n$  we will equip $BV(\R^n)$ with the homogeneous norm given by
\begin{equation}
\label{normaenbvtodoelespacio}
\norm{u}_{BV(\R^n)}= \norm{D u} (\R^n).
\end{equation}
Another $BV$-like space is $BV_{\frac{n}{n-1}}(\R^n)$, defined as the space of all functions in $L^{\frac{n}{n-1}}(\R^n)$ such that $Du$ is a finite vector-valued measure. The space $BV_{\frac{n}{n-1}}(\R^n)$ is a Banach space when equipped with the norm
\begin{equation*}
\norm{u}_{BV_{\frac{n}{n-1}}(\R^n)}= \norm{D u} (\R^n).
\end{equation*}  
\end{definition}

\begin{remark}
$BV(\R^n)$ is not a Banach space under the norm \eqref{normaenbvtodoelespacio}. Also, we have 
$$\norm{D u} (\O)=\sup\left\{ \int_\O u \, \div \varphi dx: \varphi\in C_c^1(\O) \text{ and } |\varphi(x)|\leq 1\, \forall x\in\O\right\},$$
where $\varphi=(\varphi_1, \varphi_2, ..., \varphi_n)$  and $|\varphi(x)|=(\varphi_1(x)^2 + \varphi_2(x)^2 + \cdots +\varphi_n(x)^2)^{1/2}$. In what follows, we shall also write $\int_\O |Du|$ instead of  $\norm{D u} (\O)$.
\end{remark}

We will use the following Sobolev's inequality for functions in $BV(\R^n)$  whose proof can be found in \cite[Theorem 3.47]{afp}:

\begin{theorem}
\label{mesorprendi}
Let $u \in BV(\R^n)$. Then
\begin{equation}
\label{muyutil}
\norm{u}_{L^{\frac{n}{n-1}}(\R^n)} \leq C(n) \norm{Du}(\R^n).
\end{equation}
\end{theorem}

Inequality \eqref{muyutil} immediately implies the following continuous embedding
\begin{equation}
\label{inmerso}
BV(\R^n) \hookrightarrow BV_{\frac{n}{n-1}}(\R^n).
\end{equation}

We recall that the standard Sobolev space $W^{1,1}(\O)$ is defined as the space of all functions $u \in L^1(\O)$ such that $Du \in L^1(\O)$. The Sobolev space $W^{1,1}(\O)$ is a Banach space with the norm
\begin{equation}
\label{normaenw11}
\norm{u}_{W^{1,1}(\O)}= \norm{u}_{L^1(\O)} + \norm{D u} _{L^1(\O)}= \int_{\O} \left[|u| +(|D_1u|^2 + |D_2u|^2 +\dots+ |D_nu|^2)^{\frac{1}{2}}\right] dx.
\end{equation}

However, we will often refer to the following homogeneous Sobolev space. Hereafter, we let $C_c^{\infty}(\O)$ 
denote the space of smooth functions with compact support in a general  open set $\O$.
\begin{definition}
\label{homogeneous}
Let $\dot{W}^{1,1}(\R^n)$ denote the space of all functions $ u \in L^{\frac{n}{n-1}}(\R^n)$ such that $Du \in L^1(\R^n)$. Equivalently, the space $\dot{W}^{1,1}(\R^n)$ can also be defined as the closure of $C_c^{\infty}(\R^n)$ in $BV_{\frac{n}{n-1}}(\R^n)$ (i.e., in the norm $\norm{D u}_{L^1(\R^n)})$. Thus, $u \in \dot{W}^{1,1}(\R^n)$ if and only if there exists a sequence $u_k \in C_c^{\infty}(\R^n)$ such that $\int_{\R^n}|D(u_k-u)|dx=0$, and moreover,
$$\dot{W}^{1,1}(\R^n) \hookrightarrow BV_{\frac{n}{n-1}}(\R^n).$$ 
\end{definition}

\begin{definition}
Given a bounded open set $\O$, we say that the boundary $\partial\O$ is Lipschitz if for each $x\in \partial\O$, there exist $r>0$ and a Lipschitz mapping 
$h:\R^{n-1}\rightarrow \R$ such that, upon rotating and relabeling the coordinate axes if necessary, we have
$$\O\cap B(x,r)=\{y=(y_1, \dots, y_{n-1}, y_n): h(y_1, \dots, y_{n-1})<y_n\}\cap B(x,r).$$
\end{definition}

\begin{remark}
Let $\O$ be a bounded open set with Lipschitz boundary. 
We denote by  $W^{1,1}_0(\O)$ the Sobolev space consisting of all functions in $W^{1,1}(\O)$  with zero trace on $\partial \O$. Then it is well-known
that $C_c^{\infty}(\O)$ is dense in $W^{1,1}_0(\O)$.
One of the goals of this paper is to make precise the definition of $BV_0(\O)$, the space of all functions in $BV(\O)$ with zero trace on $\partial \O$ (see Theorem \ref{sonlomismo}). In this paper we equip the two spaces, $BV_0(\O)$ and $W_0^{1,1}(\O)$, with the equivalent norms (see Theorem \ref{sobolevbv0})   to \eqref{normaenbv} and \eqref{normaenw11}, respectively, given by
\begin{equation*}
\norm{u}_{BV_0(\O)}= \norm{D u} (\O), \quad \text{and} \quad \norm{u}_{W^{1,1}_0(\O)}= \int_{\O} |Du| dx.
\end{equation*}
\end{remark}

\begin{definition}
For any open set $\Omega$, we let  $BV_c(\O)$ denote the space of functions
 in $BV(\O)$ with  compact support in $\O$. Also, $BV^{\infty}(\O)$ and $BV_{0}^{\infty}(\O)$ denote the space of bounded functions in $BV(\O)$ and $BV_0(\O)$, respectively. Finally, 
$BV_c^{\infty}(\O)$ is the space of all bounded functions  in $BV({\O})$ with compact support in $\O$. 
\end{definition}

We will use the following result (see \cite[Proposition 1.13]{Giusti}). We include the  proof here for the sake of completeness. 
\begin{lemma}
\label{muyimportante-1}
Suppose $\{u_k\}$ is a sequence in $BV(\O)$ such that $u_k \to u$ in $L^1_{loc}(\O)$ and
\begin{equation}
\label{loquese}
\lim_{k \to \infty} \int_{\O}|Du_k| =\int_{\O}|Du|.
\end{equation}
Then for every open set $A \subset \O$,
\begin{equation*}
\int_{\overline{A} \cap \O}|Du| \geq \limsup_{k \to \infty} \int_{\overline{A} \cap \O}|Du_k|.
\end{equation*}
In particular, if $\int_{\partial A \cap \O}|Du|=0$, then
\begin{equation}
\label{clave}
\int_{A}|Du| = \lim_{k \to \infty} \int_{A}|Du_k|.
\end{equation}
\end{lemma}
\begin{proof}
Consider the open set $B= \O \setminus \overline{A}$. Since $u_k \to u$ in $L^1_{loc}(\O)$, by the lower semicontinuity property we have
\begin{equation}
\label{lower}
\int_{A}|Du| \leq \liminf_{k \to \infty} \int_{A}|Du_k|, \tn{ and } \int_{B}|Du| \leq \liminf_{k \to \infty} \int_{B}|Du_k|.
\end{equation}
On the other hand,
\begin{eqnarray}
\int_{\overline{A} \cap \O}|Du|+ \int_{B}|Du| &=& \int_{\O}|Du| \nonumber\\
 &=& \lim_{k \to \infty} \int_{\O}|Du_k|=\lim_{k \to \infty} \left(\int_{\overline{A} \cap \O}|Du_k| + \int_{B}|Du_k|\right), \tn{ by } \eqref{loquese}\nonumber\\
 &=& \limsup_{k \to \infty} \left(\int_{\overline{A} \cap \O}|Du_k| + \int_{B}|Du_k|\right)\nonumber\\
 &\geq& \limsup_{k \to \infty} \int_{\overline{A} \cap \O}|Du_k| + \liminf_{k \to \infty} \int_{B}|Du_k|\nonumber\\ 
 &\geq& \limsup_{k \to \infty} \int_{\overline{A} \cap \O}|Du_k| +  \int_{B}|Du|, \tn{ by }  \eqref{lower}, \nonumber
\end{eqnarray}
and hence
\begin{equation*}
\int_{\overline{A} \cap \O}|Du| \geq \limsup_{k \to \infty} \int_{\overline{A} \cap \O}|Du_k|.
\end{equation*}
In particular, if $\int_{\partial A \cap \O}|Du| =0$ then we obtain from the last inequality
\begin{equation*}
\int_{A} |Du| =\int_{\overline{A} \cap \O}|Du| \geq \limsup_{k \to \infty} \int_{\overline{A} \cap \O}|Du_k| \geq \limsup_{k \to \infty} \int_{A}|Du_k| \geq \liminf_{k \to \infty} \int_{A}|Du_k|,
\end{equation*}
and since, by \eqref{lower},
\begin{equation*}
\int_{A} |Du| \leq  \liminf_{k \to \infty} \int_{A}|Du_k|\leq \limsup_{k \to \infty} \int_{A}|Du_k|,
\end{equation*}
clearly \eqref{clave} follows.
\end{proof}

The following theorem from functional analysis (see \cite[Theorem 1.7]{Reed} ) will be fundamental in this paper:

\begin{theorem}
\label{BLT}
Let $X$ be a normed linear space and $Y$ be a Banach space. Suppose $T:D \to Y$   is a bounded linear transformation, where
$D \subset X$ is a dense linear subspace. Then $T$ can be uniquely extended to a bounded linear transformation $\hat{T}$ from $X$ to $Y$.  In addition, the operator norm of $T$ is $c$ if and only if the norm of $\hat{T}$ is $c$.
\end{theorem}

The following formula will be important in this paper.
\begin{lemma}
\label{laformula}
Let $\mu \in \mathcal{M}_{loc}(\R^n)$ and  $f$ be a function such that $\int_{\R^n} |f| d\norm{\mu}<+\infty$. Then
\begin{equation*}
\int_{\R^n} f d\mu= \int_{0}^{\infty} \mu(\{f\geq t\})dt - \int_{-\infty}^{0} \mu(\{ f \leq t \})dt.
\end{equation*}
The same equality also holds if we replace the sets $\{f\geq t\}$ and $\{f\leq t\}$ by $\{f> t\}$ and $\{f< t\}$, respectively.
\end{lemma}
\begin{proof}
We write $f= f^{+}-f^{-}$, where $f^{+} \geq 0$ and $f^{-} \geq 0$ are the positive and negative parts of $f$. Then
\begin{eqnarray*}
\int_{\R^n} f d\mu &= & \int_{\R^n} (f^{+} - f^{-}) d\mu \\
   &= & \int_{0}^{\infty} \mu(\{f^{+}\geq t\})dt - \int_{0}^{\infty}\mu(\{ f^{-} \geq t \})dt \\
   &= & \int_{0}^{\infty} \mu(\{f \geq t \})dt - \int_{0}^{\infty}\mu(\{ -f  \geq t \})dt \\
   & =& \int_{0}^{\infty} \mu(\{f \geq t\})dt - \int_{0}^{\infty}\mu(\{ f \leq -t \})dt \\
   &=&  \int_{0}^{\infty} \mu(\{f \geq t\})dt - \int_{-\infty}^{0}\mu(\{ f \leq  s\})ds, \text{ by making the change of variables $t=-s$},    
\end{eqnarray*}
which is the desired result.
\end{proof}

\section{$BV_c^{\infty}(\R^n)$ is dense in $BV_{\frac{n}{n-1}}(\R^n)$}

\begin{theorem}\label{density}
Let $u\in BV_{\frac{n}{n-1}}(\R^n)$, $u\geq 0$, and $\phi_k\in C_c^\infty(\R^n)$ be
a nondecreasing sequence of
smooth functions satisfying: 
\begin{equation}
\label{parece1}
0 \leq \phi_k \leq 1,\, \phi_k\equiv 1 \textnormal{ on } B_k(0),\, \phi_k\equiv 0 \textnormal{ on } \R^n \setminus B_{2k}(0)
\textnormal{ and }|{D \phi_k}|\leq c/k.
\end{equation}
Then \begin{equation}\label{approxg}
\lim_{k\rightarrow\infty}\norm{(\phi_k u)- u}_{BV_{\frac{n}{n-1}}(\R^n)}=0,
\end{equation}
and for each fixed $k>0$ we have 
\begin{equation}\label{approxg-j}
\lim_{j\rightarrow\infty}\norm{(\phi_k u)\wedge j- \phi_k u}_{BV_{\frac{n}{n-1}}(\R^n)}=0.
\end{equation}

In particular, $BV_c^\infty(\R^n)$ is dense in $BV_{\frac{n}{n-1}}(\R^n)$.
\end{theorem}
\begin{proof}
 As $BV_{\frac{n}{n-1}}(\R^n)\subset BV_{loc}(\R^n)$, the product rule for $BV_{loc}$ functions gives that 
  $D(\phi_ku)=\phi_k D u + u D \phi_k$ (as measures) (see \cite[Proposition 3.1]{afp}) and hence $\phi_ku\in BV(\R^n) \subset BV_{\frac{n}{n-1}}(\R^n)$. Thus 
\begin{eqnarray}
&&\int_{\R^n}|D(u\phi_k-u)|= \int_{\R^n}|\phi_k Du-Du +uD\phi_k| \nonumber\\
&& \leq \int_{\R^n}|\phi_k-1||Du| + \int_{\R^n \cap \supp (D\phi_k)}|u||D\phi_k| \nonumber\\
&& \leq \int_{\R^n}|\phi_k-1||Du| + \frac{c}{k} \int_{B_{2k}\setminus B_k}|u| \nonumber \\
&& \leq \int_{\R^n}|\phi_k-1||Du| + \frac{c}{k} \left(\int_{B_{2k} \setminus B_k}|u|^{\frac{n}{n-1}}\right)^{\frac{n-1}{n}}|B_{2k}\setminus B_k|^{\frac{1}{n}} \nonumber\\
&& \leq \int_{\R^n}|\phi_k-1||Du| + c \left(\int_{B_{2k}\setminus B_k}|u|^{\frac{n}{n-1}}\right)^{\frac{n-1}{n}}. \label{bienagotada}
\end{eqnarray}
We let $k\to\infty$ in \eqref{bienagotada} and use \eqref{parece1} and  the dominated convergence theorem together with the fact that $u \in L^{\frac{n}{n-1}}$
to obtain \eqref{approxg}.

On the other hand, the coarea formula for $BV$ functions yields
\begin{eqnarray*}
\int_{\R^n}|D(\phi_k u- (\phi_k u)\wedge j )|&=&\int_0^\infty \mathcal{H}^{n-1}(\partial^*\{\phi_k u- (\phi_k u)\wedge j >t\})dt\\
&=&\int_0^\infty \mathcal{H}^{n-1}(\partial^*\{\phi_k u -j>t\})dt\\
&=&\int_0^\infty \mathcal{H}^{n-1}(\partial^*\{\phi_k u>j+t\})dt\\
&=&\int_j^\infty \mathcal{H}^{n-1}(\partial^*\{\phi_k u>s\})ds.
\end{eqnarray*}
Here $\partial^* E$ stands for the reduced boundary of a set $E$.
Since $\int_0^\infty\mathcal{H}^{n-1}(\partial^*\{\phi_k u>s\})ds < \infty$, the Lebesgue dominated convergence theorem yields 
the limit \eqref{approxg-j} for each fixed $k>0$.

By the triangle inequality and \eqref{approxg}-\eqref{approxg-j}, each nonnegative     $ u\in BV_{\frac{n}{n-1}}(\R^n)$ 
can be approximated by a function in $BV_c^\infty(\R^n)$. For a general $ u\in BV_{\frac{n}{n-1}}(\R^n)$, let $u^{+}$ be the positive part of $u$. From the proof of \cite[Theorem 3.96]{afp}, we have  $u^{+}\in BV_{loc}(\R^n)$ and $\norm{Du^{+}}(A)\leq \norm{Du}(A)$ for any open set $A\Subset\R^n$. Thus $\norm{Du^{+}}(\R^n)\leq \norm{Du}(\R^n)<+\infty$ and  $u^+$ belongs to  $BV_{\frac{n}{n-1}}(\R^n)$. Likewise, we have $u^-\in BV_{\frac{n}{n-1}}(\R^n)$. Now
 by considering  separately the positive and negative parts of a function $u\in BV_{\frac{n}{n-1}}(\R^n)$, it is then easy to see 
the density of $BV_c^\infty(\R^n)$  in $BV_{\frac{n}{n-1}}(\R^n)$.
\end{proof}
We have the following corollaries of Theorem \ref{density}:

\begin{corollary}
$BV_c^{\infty}(\R^n)$ is dense in $BV(\R^n)$.
\end{corollary}
\begin{proof}
This follow immediately from \eqref{inmerso} and Theorem \ref{density}.
\end{proof}

\begin{corollary}
\label{iso-R}
The spaces $BV(\R^n)^*$ and $BV_{\frac{n}{n-1}}(\R^n)^*$ are isometrically isomorphic.
\end{corollary}
\begin{proof}
We define the map
\begin{equation*}
S:BV_{\frac{n}{n-1}}(\R^n)^* \to BV(\R^n)^*
\end{equation*}
as 
\begin{equation*}
S(T)= T \rtangle BV(\R^n).
\end{equation*}
First, we note the $S$ is injective since $S(T)=0$ implies that $T \rtangle BV(\R^n) \equiv 0$. In particular, $T \rtangle BV_c^{\infty}(\R^n) \equiv 0$. Since $BV_c^{\infty}(\R^n)$ is dense in $BV_{\frac{n}{n-1}}(\R^n)$ and $T$ is continuous on $BV_{\frac{n}{n-1}}(\R^n)$, it is easy to see that 
 $T \rtangle BV_{\frac{n}{n-1}}(\R^n)\equiv 0$. We now proceed to show that $S$ is surjective. Let $T \in BV(\R^n)^*$. Then $T \rtangle BV_c^{\infty}(\R^n)$ is a continuous linear functional. Using again that $BV_c^{\infty}(\R^n)$ is dense in $BV_{\frac{n}{n-1}}(\R^n)$, $T \rtangle BV_c^{\infty}(\R^n)$ has a unique continuous extension $\hat{T} \in BV_{\frac{n}{n-1}}(\R^n)^*$ and clearly $S(\hat{T})= T$. 
Moreover, for any $T \in BV(\R^n)^*$, the unique extension $\hat{T}$ to $BV_{\frac{n}{n-1}}(\R^n)$ has the same norm (see Theorem \ref{BLT}), that is,
\begin{equation*}
\norm{T}_{BV(\R^n)^*}= \norm{\hat{T}}_{BV_{\frac{n}{n-1}}(\R^n)^*},
\end{equation*}
and hence
\begin{equation*}
\norm{S(\hat{T})}_{BV(\R^n)^*}= \norm{\hat{T}}_{BV_{\frac{n}{n-1}}(\R^n)^*},
\end{equation*}
which implies that $S$ is an isometry.
\end{proof}

We now proceed to make precise our definitions of measures in  $\dot{W}^{1,1}(\R^n)^*$
and $BV_{\frac{n}{n-1}}(\R^n)^*$.
\begin{definition}
\label{panza1}
We let
\begin{equation*}
\mathcal{M}_{loc} \cap \dot{W}^{1,1}(\R^n)^*:= \{T \in \dot{W}^{1,1}(\R^n)^* : T(\varphi)=\int_{\R^n} \varphi d\mu \text{ for some } \mu \in \mathcal{M}_{loc}(\R^n), \forall \varphi \in C_c^{\infty}(\R^n)\}. 
\end{equation*}
Therefore, if $\mu \in \mathcal{M}_{loc}(\R^n) \cap \dot{W}^{1,1}(\R^n)^*$, then the action $<\mu,u>$ can be uniquely defined for all $u \in \dot{W}^{1,1}(\R^n)$ (because of the density of $C_c^{\infty}(\R^n)$ in $\dot{W}^{1,1}(\R^n))$.
\end{definition}

\begin{definition}
\label{panza2}
We let
\begin{equation*}
\mathcal{M}_{loc} \cap BV_{\frac{n}{n-1}}(\R^n)^*:=\{T \in BV_{\frac{n}{n-1}}(\R^n)^*: T(\varphi)=\int_{\R^n} \varphi^* d\mu \text{ for some } \mu \in \mathcal{M}_{loc}, \forall \varphi \in BV_c^{\infty}(\R^n)\},
\end{equation*} 
where $\varphi^*$ is the precise representative of $\varphi$ in $BV_c^{\infty}(\R^n)$. Thus, if $\mu \in \mathcal{M}_{loc} \cap BV_{\frac{n}{n-1}}(\R^n)^*$, then the action $<\mu,u>$ can be uniquely defined for all $u \in BV_{\frac{n}{n-1}}(\R^n)$ (because of the density of $BV_c^{\infty}(\R^n)$
in $BV_{\frac{n}{n-1}}(\R^n))$.
\end{definition}

We will study the normed linear spaces $\mathcal{M}_{loc} \cap \dot{W}^{1,1}(\R^n)^*$
and $\mathcal{M}_{loc} \cap BV_{\frac{n}{n-1}}(\R^n)^*$ in the next section. In particular, we will show in Theorem \ref{nuevo} below that these spaces are isometrically isomorphic.  In Definition \ref{panza2}, if we use $C_c^{\infty}(\R^n)$ instead of $BV_c^{\infty}(\R^n)$, then by the Hahn-Banach Theorem there exist a non-zero $T \in BV_{\frac{n}{n-1}}(\R^n)^*$ that is represented by the zero measure, which would cause a problem of injectivity in Theorem \ref{nuevo}.


\section{Characterizations of measures in $BV_{\frac{n}{n-1}}(\R^n)^*$}
The following lemma characterizes all the distributions in $\dot{W}^{1,1}(\R^n)^*$. We recall that $\dot{W}^{1,1}(\R^n)$ is the homogeneous Sobolev space introduced in Definition \ref{homogeneous}.  
\begin{lemma}
\label{lemma1}
The distribution $T$ belongs to $\dot{W}^{1,1}(\R^n)^*$ if and only if $T=\div \FF$ for some vector field $\FF \in L^{\infty}(\R^n,\R^n)$. Moreover,
\begin{equation*}
 \norm{T}_{\dot{W}^{1,1}(\R^n)^*}= \min \{ \norm{\FF}_{L^{\infty}(\R^n,\R^n)} \},
\end{equation*}
where the minimum is taken over all $\FF \in L^{\infty}(\R^n,\R^n)$ such that $\div \FF =T$. Here we use the norm
$$\norm{\FF}_{L^{\infty}(\R^n,\R^n)}:=  \norm{(F_1^2 + F_2^2 + \cdots + F_n^2)^{1/2}}_{L^{\infty}(\R^n)} \text{ for } \FF=(F_1, \dots, F_n).$$
\end{lemma}

\begin{proof}
It is easy to see that if $T= \div \FF$ where $\FF \in L^{\infty}(\R^n,\R^n)$ then $T \in \dot{W}^{1,1}(\R^n)^*$ with
\begin{equation*}
\norm{T}_{\dot{W}^{1,1}(\R^n)^*} \leq \norm{\FF}_{L^{\infty}(\R^n,\R^n)}.
\end{equation*} 
Conversely, let $T \in \dot{W}^{1,1}(\R^n)^*$. Define
\begin{equation*}
A:\dot{W}^{1,1}(\R^n)\to L^1(\R^n,\R^n), \quad A(u)=D u,
\end{equation*}
and note that the range of $A$ is a closed subspace of $L^{1}(\R^n,\R^n)$ since $\dot{W}^{1,1}(\R^n)$ is complete.
We denote the range of $A$ by $R(A)$ and we define
\begin{equation*}
T_1:R(A) \to \R
\end{equation*}
as
\begin{equation*}
T_1(Du)=T(u), \text{ for each } Du \in R(A).
\end{equation*}
Then we have
\begin{equation*}
\norm{T_1}_{R(A)^*} = \norm{T}_{\dot{W}^{1,1}(\R^n)^*}.
\end{equation*}
By Hahn-Banach Theorem there exists a norm-preserving extension $T_2$ of $T_1$ to all $L^1(\R^n,\R^n)$. On the other hand, by the Riesz Representation Theorem for vector valued functions (see \cite[pp. 98--100]{DU}) there exists a vector field $\FF \in L^{\infty}(\R^n,\R^n)$ such that 
\begin{equation*}
T_2(v)= \int_{\R^n} \FF \cdot v, \text{ for every } v \in L^{1}(\R^n,\R^n),
\end{equation*}
and
\begin{equation*}
 \norm{\FF}_{L^{\infty}(\R^n,\R^n)}= \norm{T_2}_{L^1(\R^n,\R^n)^*}= \norm{T_1}_{R(A)^*}=\norm{T}_{\dot{W}^{1,1}(\R^n)^*}.
\end{equation*}
In particular, for each $\varphi \in C_c^{\infty}(\R^n)$ we have
\begin{equation*}
  T(\varphi)= T_1(D \varphi)= T_2(D\varphi) = \int_{\R^n} \FF \cdot D\varphi,
\end{equation*}
which yields
\begin{equation*}
 T= \div (-\FF),
\end{equation*}
with
\begin{equation*}
\norm{-\FF}_{L^{\infty}(\R^n,\R^n)}= \norm{T}_{\dot{W}^{1,1}(\R^n)^*}.
\end{equation*}
\end{proof}

\begin{theorem}
\label{pmcondicion}
Let $\O \subset \R^n$ be any open set and suppose $\mu \in \M_{loc}(\O)$  such that
\begin{equation} \label{pmcondition}
|\mu (U)|\leq C\, \H^{n-1} (\partial U)
\end{equation}
for any smooth open and bounded set $U\subset \subset \O$.
Let $A$ be a compact set of $\O$.
If $\H^{n-1} (A)=0$, then $\mu(A)=0$.
\end{theorem}

\begin{proof}
As $\H^{n-1}(A)=0$, for any $0 < \varepsilon  <  \frac{1}{2}\textnormal{dist}(A, \partial  \O)$ (or for any $\varepsilon > 0$, if $\O=\R^n$), we can find a finite number of balls $B(x_i,r_i)$, $i\in I$, with 
$2r_i<\varepsilon$ such that $A\subset\bigcup\limits_{i\in I}B(x_i,r_i) \subset \O$ and 
\begin{equation} 
\label{estrella2}
\sum_{i\in I} r_i^{n-1} < \varepsilon.
\end{equation}
Let $W_\varepsilon=\bigcup\limits_{i\in I} B(x_i,r_i)$.
Then 
\begin{equation*}
A\subset\subset W_\varepsilon\subset A_\varepsilon\colon = \{x\in\Bbb R^n\colon \textrm{dist}(x,A) < \varepsilon\}.
\end{equation*}
The first inclusion follows since $A$ is compact and $W_\varepsilon$ is open; the second one follows since $2r_i<\varepsilon$ and since we may assume that $B(x_i,r_i)\cap A\neq \emptyset$ for any $i\in I$.

\noindent
We now claim that for each $\varepsilon > 0$ there exists an open set $W'_\varepsilon$ such that $W'_\varepsilon$ has smooth boundary and
\begin{equation}
\label{estrella}
\begin{cases}
A\subset\subset  W'_\varepsilon \subset A_{2\varepsilon}\\
\H^{n-1} (\partial W'_\varepsilon)\leq P(W_\varepsilon,\O),\end{cases}
\end{equation}
where $P(E,\O)$ denotes  the perimeter of a set $E$ in $\O$.  Assume for now that  $\eqref{estrella}$ holds. Then, since $A$ is compact,
\begin{equation*}
\chi_{W'_\varepsilon} {\rightarrow}{\chi_A}    \tn{ pointwise  as } \varepsilon\to 0,
\end{equation*}
and 
\begin{eqnarray*}
|\mu (W'_\varepsilon)|&\leq& C \H^{n-1} (\partial W'_\varepsilon), \tn{ by our hypothesis }\eqref{pmcondition}\\
&\leq& C P(W_\varepsilon,\O)\\
&\leq& C \sum_{i\in I} r_i^{n-1}\leq \varepsilon C, \textnormal{ by } \eqref{estrella2}.
\end{eqnarray*}
Thus, the Lebesgue dominated convergence theorem yields, after letting $\varepsilon\to 0$, the desired result:
\begin{equation*}
|\mu (A)|=0.
\end{equation*}

\noindent
We now proceed to prove \eqref{estrella}. Let $\rho$ be a standard symmetric mollifier:
\begin{equation*}
\rho\geq 0,\  \rho\in C_0^\infty (B(0,1)),\  \int_{\Bbb R^n} \rho(x) dx=1,\ \text{and } \rho(x)=\rho(-x). 
\end{equation*}
Define $\rho_{1/ k}(x)=k^n \rho(k x)$ and 
\begin{equation*}
u_k(x)=\chi_{W_\varepsilon}*\rho_{1/ k}(x)=k^n\int\rho(k(x-y))\chi_{W_\varepsilon}(y) dy
\end{equation*}
for $k=1,2,\ldots$ For $k$ large enough, say for $k\geq k_0=k_0(\epsilon)$, it follows that
\begin{eqnarray}
u_k\equiv 1\textrm{ on }A, \tn{ since } A\subset\subset W_\varepsilon \label{arriba2},\\
u_k\equiv0 \tn{ on }  \O \backslash A_{2\varepsilon}, \textrm{ since } W_\varepsilon\subset A_\varepsilon \label{abajo}.
\end{eqnarray}
We have
\begin{eqnarray*}
P(W_\varepsilon,\O)&=&|D\chi_{W_\varepsilon}|(\O)\\
&\geq& |D u_k|(\O)\\
&=& \int_0^1 P (F_t^k,\O) dt, \tn{ since } 0 \leq u_k \leq 1,
\end{eqnarray*}
where
\begin{equation*}
F_t^k := \{x\in\O\colon u_k (x) > t\}.
\end{equation*}
Note that for $k\geq k_0$, and $t\in (0,1)$ we have, by \eqref{arriba2} and \eqref{abajo},
\begin{equation*}
A\subset\subset F_t^k\subset A_{2\varepsilon}.
\end{equation*}
For a.e.~$t\in (0,1)$ the sets $F_t^k$ have smooth boundaries.
Thus we can choose $t_0 \in(0,1)$ with this property and such that
\begin{equation*}
P(F_{t_0}^k,\O)\leq P(W_\varepsilon,\O),
\end{equation*}
which is 
\begin{equation*}
\H^{n-1} (\partial F_{t_0}^k)\leq P(W_\varepsilon,\O).
\end{equation*}

Finally, we choose $W'_\varepsilon=F^k_{t_0}$ for any fixed $k\geq k_0$.
\end{proof}

\begin{corollary}
\label{pmcondition2}
If $\mu \in \M_{loc}(\O)$ satisfies the hypothesis of Theorem \ref{pmcondicion}, then $\norm{\mu} << \H^{n-1}$ in $\Omega$; that is, if $A \subset \Omega$ is any Borel measurable set such that $\H^{n-1}(A)=0$ then $\norm{\mu}(A)=0$.
\end{corollary}

\begin{proof}
The domain $\O$ can be decomposed as $\O=\O^{+}\cup \O^{-}$, such that $\mu^{+}=\mu \rtangle \O^{+}$  and $\mu^{-}=\mu \rtangle \O^{-}$, where
$\mu^{+}$ and $\mu^{-}$ are the positive and negative parts of $\mu$, respectively. Let $A \subset \O$ be a Borel set satisfying $\H^{n-1}(A)=0$. By writing 
$A=(A\cap\O^+) \cup (A\cap\O^-)$, we may assume that $A \subset \O^{+}$ and hence $\norm{\mu}(A)=\mu^{+}(A)$. Moreover, since $\mu^{+}$ is a Radon measure we can assume that $A$ is compact. Hence, Theorem \ref{pmcondicion} yields $\norm{\mu}(A) = \mu^{+}(A)=\mu(A)=0$.
\end{proof}

The following theorem characterizes all the signed measures in $BV_{\frac{n}{n-1}}(\R^n)^*$. This result was first proven in Phuc-Torres \cite{PT} for the space $BV(\R^n)^*$
with no sharp control on the involving constants.  In this paper we offer a new and direct proof of ${\bf (i)} \Rightarrow {\bf (ii)}$. We also clarify the first part of ${\bf (iii)}$. Moreover, our proof of ${\bf (ii)} \Rightarrow {\bf (iii)}$ yields a sharp constant that will be needed for the proof of Theorem \ref{nuevo} below.

\begin{theorem}\label{superestrella}
 Let $\mu \in \mathcal{M}_{loc} (\rnm)$ be a locally finite signed measure. The following are equivalent:
 
 {\rm {\bf(i)}} There exists a vector field $\FF \in L^{\infty}(\rnm,\rnm)$ such that
 $\div \FF = \mu$ in the sense of distributions.

{\rm {\bf(ii)}} There is a constant $C$ such that 
$$| \mu(U)| \leq C\, \H^{n-1}(\po U)$$
 for any smooth bounded open (or closed) set $U$ with $\H^{n-1}(\po U) < +\infty$. 

 {\rm {\bf(iii)}} $\H^{n-1}(A)=0$ implies $\norm{\mu}(A)=0$ for all Borel sets $A$ and there is a constant $C$ such that, for all $ u \in BV_c^\infty(\rnm)$,
 $$
|<\mu,u>|:=\left|\int_{\rnm} u^*d\mu\right| \leq C\int_{\rnm}| D u|, 
$$
where $u^*$ is the representative in the class of $u$ that is defined
$\H^{n-1}$-almost everywhere.

{\rm {\bf(iv)}}  $\mu \in BV_{\frac{n}{n-1}}(\rnm)^*$. The action of $\mu$ on any $u \in BV_{\frac{n}{n-1}}(\rnm)$
is defined (uniquely) as
$$ 
<\mu, u>:= \lim_{k \to \infty} <\mu,u_k>= \lim_{k \to \infty} \int_{\rnm} u_{k}^*d\mu,
$$
where $u_k \in BV_c^\infty(\rnm)$ 
converges to $u$ in $BV_{\frac{n}{n-1}}(\rnm)$. In particular, if $u \in BV_c^{\infty}(\rnm)$ then
$$<\mu,u>=\int_{\rnm} u^*d\mu,$$
and moreover, if $\mu$ is a non-negative measure then,
for all $u \in BV_{\frac{n}{n-1}}(\rnm)$,
$$
<\mu,u>=\int_{\rnm} u^*d\mu.
$$

 
 
 \end{theorem}
 \begin{proof}  Suppose {\bf(i)} holds.  Then for every $\varphi \in C_c^{\infty}(\R^n)$ we have
 \begin{equation}
 \label{manita1}
 \int_{\R^n} \FF \cdot D \varphi dx = -\int_{\R^n} \varphi d\mu.
 \end{equation}
 Let $U \subset \subset \R^n$ be any open set (or closed set) with smooth boundary satisfying $\H^{n-1}(\partial U) < \infty$. Consider the characteristic function $\chi_U$ and a sequence of mollifications
\begin{equation*}
u_k:= \chi_U * \rho_{1/k},
\end{equation*}   
where $\{\rho_{1/k}\}$ is as in the proof of Theorem \ref{pmcondicion}. Then, since $U$ has a smooth boundary, we have 
\begin{equation}
\label{manota}
u_k(x) \to \chi_U^*(x) \text{ pointwise everywhere},
\end{equation}
where $\chi_U^*(x)$ is the precise representative of $\chi_U$ (see \cite[Corollary 3.80]{afp} given by 
\[
\chi_U^*(x)=
\begin{cases}
1&\text{$x\in \textnormal{Int}(U)$,}\\
\frac{1}{2}&\text{$x\in\partial U$,}\\
0&\text{$x\in \R^n \setminus \overline{U}$.}
\end{cases}
\]
We note that $\chi_U^*$ is the same for  $U$ open or closed, since both are the same set of finite perimeter (they differ only on $\partial U$, which is  a set of Lebesgue measure zero).
From \eqref{manita1}, \eqref{manota}, and the dominated convergence theorem we obtain
\begin{eqnarray}
\label{marisol}
\left|\mu ({\rm Int}(U)) + \frac{1}{2} \mu (\partial U)\right| &=& \left|\int_{\R^n} \chi_U^* d\mu \right|=  \lim_{k \to \infty} \left|\int_{\R^n} u_k d\mu \right|\\ 
&=& \lim_{k \to \infty} \left| \int_{\R^n} \FF \cdot D u_k dx \right| \nonumber \\
&\leq & \lim_{k\rightarrow\infty} \norm{\FF}_{\infty} \int_{\R^n} |D u_k| dx \nonumber \\
&= & \norm{\FF}_{\infty} \int_{\R^n} |D \chi_U| = \norm{\FF}_{\infty} \H^{n-1}(\partial U). \nonumber 
\end{eqnarray}
We now let 
\begin{equation*}
K:= \overline{U}.
\end{equation*}
For each $h>0$ we define the function
\begin{equation*}
F_h(x)= 1 - \frac{\min\{d_{K}(x),h\}}{h}, \qquad x \in \R^n,
\end{equation*}
where $d_{K}(x)$ denotes the distance from $x$ to $K$, i.e., $d_{K}(x)=\inf\{|x-y|: y\in K\}$.
 Note that $F_h$ is a Lipschitz function such that $F_h(x) \leq 1$, $F_h(x)=1$ if $x \in  K$ and $F_h(x)=0$ if $d_{K}(x) \geq h$. Moreover, $F_h$ is differentiable $\mathcal{L}^n$-almost everywhere and
\begin{equation*}
|DF_h(x)| \leq \frac{1}{h} \text{ for } \mathcal{L}^n\text{-a.e. } x \in \R^n.
\end{equation*}
By standard smoothing techniques, \eqref{manita1} holds for the Lipschitz function $F_h$. Therefore,
\begin{eqnarray}
\label{pata1}
\left| \int_{\R^n} F_h d\mu \right|&= & \left| \int_{\R^n} \FF \cdot DF_h dx \right|. 
\end{eqnarray}
Since $F_h \to \chi_{K}$ pointwise, it follows from the dominated convergence theorem that
\begin{equation}
\label{pata2}
|\mu (K)| = \left|\int_{\R^n} \chi_{K} d\mu \right|=  \lim_{h \to 0} \left|\int_{\R^n} F_h d\mu \right|.
\end{equation}
On the other hand, using the coarea formula for Lipschitz maps, we have
\begin{eqnarray}
\label{pata3}
 \left| \int_{\R^n} \FF \cdot DF_h dx\right|  &\leq& \norm{\FF}_{\infty}  \int_{\R^n}| DF_h| dx \\
&= &  \norm{\FF}_{\infty}\frac{1}{h} \int_{\{0 < d_{ K} < h \} } |D d_{ K}| dx \nonumber \\  
&= & \norm{\FF}_{\infty} \frac{1}{h} \int_{0}^{h} \H^{n-1} (d_{ K}^{-1} (t) ) dt \nonumber \\
&=& \norm{\FF}_{\infty} \H^{n-1} (d_{ K}^{-1} (t_e^h))\nonumber,
\end{eqnarray}
where $0<t^h_e<h$, and $d_{ K}^{-1} (t_e^h) \subset (\R^n \setminus K)$. Because $K$ is smoothly bounded, it follows that
\begin{equation}
\label{pata4}
\H^{n-1} (d_{ K}^{-1}(t_e^h)) \to \H^{n-1}(\partial K)  \text{ as } h \to 0.
\end{equation}
Since $K= \overline{U}$ and $\partial K= \partial U$, it follows from \eqref{pata1}-\eqref{pata4} that 
\begin{equation}
\label{johnpaul}
|\mu(\overline{U})| \leq \norm{\FF}_{\infty} \H^{n-1}(\partial U).
\end{equation}
From \eqref{marisol} and \eqref{johnpaul} we conclude that, for any open set (or closed) $U \subset \subset \R^n$ with smooth boundary and finite perimeter,
$$\frac{1}{2}|\mu(\partial U)|= \left| \mu(\overline{U}) -[\mu({\rm Int}(U)) + \frac{1}{2}\mu(\partial U)]\right| \leq 2\norm{\FF}_{\infty} \H^{n-1}(\partial U),$$
and hence
\begin{equation*}
|\mu({\rm Int}(U))| \leq  3\norm{\FF}_{\infty} \H^{n-1}(\partial U).
\end{equation*}

This completes the proof of {\bf(i)} $\Rightarrow$ {\bf (ii)} with $C=\norm{\FF}_{\infty}$ for closed sets and $C=3\norm{\FF}_{\infty}$ for open sets.

 We proceed now to show that {\bf(ii)} $\Rightarrow$ {\bf(iii)}. Corollary \ref{pmcondition2} says that $\norm{\mu} << \H^{n-1}$, which proves the first part of {\bf(iii)}.  We let
 $u \in BV_c^\infty(\rnm)$ 
and we consider
 the convolutions $\rho_\ve * u$ and
define
\begin{equation*} 
A_t^\ve:= \{ \rho_\ve *u  \geq  t\} \text{ for }t>0, \text{ and } B_t^\ve:= \{ \rho_\ve *u  \leq t\} \text{ for }t<0.
\end{equation*}
Since $\rho_\ve * u \in C_c^\infty(\rnm)$
it follows that $\po A_t^\ve$ and $\po B_t^\ve $ are smooth  for a.e. $t$. Applying Lemma \ref{laformula} we compute
\begin{eqnarray}
\left|\int_{\rnm}\rho_\ve*u d\mu\right|&=&\left|\int^\infty_0\mu(A_t^\ve) dt - \int^0_{-\infty}\mu(B_t^\ve) dt\right| \nonumber  \\
&\leq & \int^\infty_0|\mu(A_t^\ve)| dt + \int^0_{-\infty}|\mu(B_t^\ve)| dt \nonumber \\
&\leq& C  \int^\infty_0 \H^{n-1}(\po A_t^\ve) dt  + C  \int^0_{-\infty} \H^{n-1}(\po B_t^\ve)\,dt \textnormal{, by } {\bf (ii)} \nonumber \\
&=& C \int_{\rnm}|D (\rho_\ve*u)|\,dx  \text {, by the Coarea Formula}\nonumber \\
& \leq&  C \int_{\rnm}|D u|. \label{masesperanza}
\end{eqnarray}

 We let $u^*$ denote the precise representative of $u$. We have that
  (see Ambrosio-Fusco-Pallara \cite{afp}, Chapter 3, Corollary 3.80):
\begin{equation}
\label{condic}
\rho_\ve* u \to u^* \quad \H^{n-1}\text{-}\tn{almost everywhere}.
\end{equation}
We now let $\ve \to 0$ in \eqref{masesperanza}. Since $u$ is bounded and $\norm{\mu} << \H^{n-1}$, 
 \eqref{condic} and the dominated convergence theorem yield
\begin{equation*}
\left|\int_{\rnm}u^* d\mu\right| \leq C\int_{\rnm}|D u|,
\end{equation*}
which completes the proof of  {\bf (ii)} $\Rightarrow$ {\bf (iii)} with the same constant $C$ as given in {\bf (ii)}.

From {\bf(iii)} we obtain that the linear operator
\begin{equation}\label{TU}
 T(u):= <\mu,u>=\int_{\rnm}u^* d\mu, \quad u \in BV_c^\infty(\rnm)
\end{equation} 
is continuous and hence it can be uniquely extended, since 
$BV_c^\infty(\rnm)$ is dense in $BV_{\frac{n}{n-1}}(\rnm)$ (Lemma \ref{density}), 
to the space $BV_{\frac{n}{n-1}}(\rnm)$. 

Assume now that $\mu$ is non-negative. We take
$u \in BV_{\frac{n}{n-1}}(\rnm)$ and consider the positive and negative parts $(u^*)^+$
and $(u^*)^-$ of the representative $u^*$. With $\phi_k$ as in Lemma \ref{density}, using \eqref{TU} we have 
$$
T([\phi_k(u^*)^+] \wedge j )= \int_{\rnm}[\phi_k(u^*)^+]\wedge j \,  d\mu, \quad j=1,2, \dots 
$$
We first let $j\rightarrow\infty$ and then   $k\rightarrow\infty$. Using Lemma \ref{density}, the continuity of $T$, and the  monotone convergence theorem we find
$$
T((u^*)^+)= \int_{\rnm}(u^*)^+ d\mu.
$$
We proceed in the same way for  $(u^*)^-$ and thus by linearity we conclude
$$
T(u)=T((u^*)^+)-T((u^*)^-)=\int_{\rnm}(u^*)^+ -(u^*)^- d\mu=\int_{\rnm}u^* d\mu.
$$
To prove that {\bf (iv)} implies {\bf (i)} we take $\mu \in BV_{\frac{n}{n-1}}(\R^n)^*$. Since $\dot{W}^{1,1}(\R^n)\subset BV_{\frac{n}{n-1}}(\R^n)$ then
\begin{equation*}
\tilde{\mu}:= \mu \rtangle \dot{W}^{1,1}(\R^n) \in \dot{W}^{1,1}(\R^n)^*, 
\end{equation*}        
and therefore Lemma  \ref{lemma1} implies that there exists $\FF \in L^\infty(\R^n,\R^n)$ such that $\div \FF =\tilde{\mu}$ and thus, since $C_c^{\infty} \subset \dot{W}^{1,1}(\R^n)$, we conclude that
$\div \FF= \mu $
in the sense of distributions. 
\end{proof}

\begin{remark}
Inequality \eqref{johnpaul} can also be obtained be means of the (one-sided) \emph{outer Minskowski content}. Indeed, since $|D d_{K}|=1$ a.e., we find  
\begin{eqnarray*}
 \left| \int_{\R^n} \FF \cdot DF_h dx\right|  &\leq& \norm{\FF}_{\infty}  \int_{\R^n}| DF_h| dx \\
&= &  \norm{\FF}_{\infty}\frac{1}{h} |\{0 < d_{ K} < h \}|.  
\end{eqnarray*}
Now sending $h\rightarrow 0^+$ and using \eqref{pata1}-\eqref{pata2} we have 
$$|\mu(K)| \leq \norm{\FF}_{\infty} \mathcal{SM}(K)= \norm{\FF}_{\infty} \mathcal{H}^{n-1}(\partial K),$$
where $\mathcal{SM}(K)$ is the outer Minskowski content of $K$ (see \cite[Definition 5]{acv}), and the last equality follows from \cite[Corollary 1]{acv}. 
This argument also holds in the case $U$ only has a Lipschitz boundary. Note that in this case we can only say that the limit in $\eqref{manota}$ holds $\mathcal{H}^{n-1}$-a.e., but this is enough for \eqref{marisol} since $\norm{\mu}<<\mathcal{H}^{n-1}$ by   \eqref{manita1} and \cite[Lemma 2.25]{ctz2}.
\end{remark}

\begin{remark}
 If  $\FF\in L^{\infty}(\rnm,\rnm)$ satisfies $\div \FF= \mu$ then, for any bounded
set of finite perimeter $E$, the Gauss-Green formula 
proved in Chen-Torres-Ziemer \cite{ctz2} yields,
\begin{equation*}
\mu(E^1 \cup \po^*E)=\int_{E^1 \cup \po^* E}\div \FF  =  \int_{\po^* E}(\mathscr{F}_e\cdot
\bnu)(y)d \mathcal{H}^{n-1}(y) 
\end{equation*}
and
\begin{equation*}
\mu(E^1 )= \int_{E^1}\div \FF  =  \int_{\po^* E}(\mathscr{F}_i \cdot
\bnu)(y)d \mathcal{H}^{n-1}(y).
\end{equation*}
Here $E^1$ is the measure-theoretic interior of $E$ and $\partial^* E$ is the reduced boundary of $E$.
The estimates $$\norm{\mathscr{F}_e\cdot \bnu}_{L^{\infty}(\po^* E)}
\leq \norm{\FF}_{L^{\infty}} {\rm ~and~} \norm{\mathscr{F}_i\cdot \bnu}_{L^{\infty}(\po^* E)}
\leq \norm{\FF}_{L^{\infty}}$$ give
\begin{equation*}
|\mu (E^1 \cup \po^*E)|=|\mu(E^1)+ \mu(\po^*E)| \leq \norm{\FF}_{L^{\infty}}\H^{n-1}(\po^* E)
\end{equation*}
and 
\begin{equation*} 
|\mu (E^1)| \leq \norm{\FF}_{L^{\infty}}\H^{n-1}(\po^* E).
\end{equation*}
Therefore, 
 $$|\mu (\po^* E)| \leq \norm{\FF}_{L^{\infty}}\H^{n-1}(\po^* E) + |\mu(E^1)| \leq
 2\norm{\FF}_{L^{\infty}}\H^{n-1}(\po^* E).$$We note  that this provides another proof of
 {\bf(i)} $\Rightarrow$ {\bf(ii)} (with $C=\norm{\FF}_{\infty}$  for both open and closed smooth sets) since for any bounded open (resp. closed) set $U$ with smooth boundary 
  we have $U=U^1$ (resp. $U=U^1\cup \partial^*U$).  
\end{remark}

We recall the spaces defined in Definitions \ref{panza1} and \ref{panza2}. We now show the following new result.
\begin{theorem}
\label{nuevo}
Let $\mathcal{E}:= \mathcal{M}_{loc} \cap BV_{\frac{n}{n-1}}(\R^n)^*$ and $\mathcal{F}:= \mathcal{M}_{loc} \cap \dot{W}^{1,1}(\R^n)^*$. Then $\mathcal{E}$ and $\mathcal{F}$ are isometrically isomorphic.
\end{theorem}
\begin{proof}
We define a map $S:\mathcal{E} \to \mathcal{F}$
as
\begin{equation*}
 S(T)=T \rtangle \dot{W}^{1,1}.
\end{equation*}
Clearly, $S$ is a linear map. We need to show that $S$ is 1-1 and on-to, and $\norm{S(T)}_{\dot{W}^{1,1}(\R^n)^*}=\norm{T}_{BV_{\frac{n}{n-1}}(\R^n)^*}$ for all
$T\in \mathcal{E}$.
In order to show the injectivity we assume that $S(T)=0 \in \mathcal{F}$ for some $T\in \mathcal{E}$. Then
\begin{equation*}
 T(u)=0 \text{ for all } u \in \dot{W}^{1,1}(\R^n).
\end{equation*}
Thus, if $\mu$ is the measure associated to $T \in \mathcal{E}$, then
\begin{equation*}
\int_{\R^n} \varphi d \mu= T(\varphi)=0 \text{ for all } \varphi \in C_c^{\infty}(\R^n),
\end{equation*}
which implies that $\mu=0$. Now, by definition of $\mathcal{E}$, we have
\begin{equation*}
T(u)=\int_{\R^n} u^* d\mu =0 \text{ for all  } u \in BV_c^{\infty}(\R^n),
\end{equation*}
which implies, by Theorem \ref{BLT} and Theorem \ref{density}, that 
\begin{equation*}
T\equiv 0 \text{ on } BV_{\frac{n}{n-1}(\R^n)}.
\end{equation*}

We now proceed to show the surjectivity and take $H \in \mathcal{F}$. Thus, there exists 
$\mu \in \mathcal{M}_{loc}(\R^n)$ such that
\begin{equation*}
\int_{\R^n} \varphi d\mu = H(\varphi)  \text{ for all } \varphi \in C_c^{\infty}(\R^n). 
\end{equation*}
From Lemma \ref{lemma1}, since $H \in \dot{W}^{1,1}(\R^n)^*$, there exists a bounded vector field $\FF \in L^{\infty}(\R^n,\R^n)$ such that
\begin{equation}
\label{aquienescojo}
\div \FF=\mu \text{ in the  distributional sense and } \norm{H}_{\dot{W}^{1,1}(\R^n)^*}=\norm {\mu}_{\dot{W}^{1,1}(\R^n)^*} = \norm{\FF}_{L^{\infty}(\R^n,\R^n)}.
\end{equation}
Now, from the proof of Theorem \ref{superestrella}, ${\bf (i)} \Rightarrow {\bf (ii)} \Rightarrow {\bf (iii)}$,  it follows that 
\begin{equation*}
\norm{\mu} << \H^{n-1},
\end{equation*}
\begin{equation*}
|\mu(U)| \leq \norm{\FF}_{\infty} \H^{n-1}(\partial U)
\end{equation*}
for all closed and smooth sets $U \subset \subset \R^n$, and 
\begin{equation*}
\left|\int_{\R^n} u^* d \mu \right| \leq \norm{\FF}_{L^{\infty}(\R^n,\R^n)} \norm{u}_{BV_{\frac{n}{n-1}}(\R^n)} \text{ for all } u \in BV_c^{\infty}(\R^n).
\end{equation*}

Hence, $\mu \in BV_c^{\infty}(\R^n)^*$ and from \eqref{aquienescojo} we obtain
\begin{equation*}
\norm {\mu}_{BV_c^{\infty}(\R^n)^*} = \norm{\FF}_{L^{\infty}(\R^n,\R^n)}=\norm {\mu}_{\dot{W}^{1,1}(\R^n)^*}.
\end{equation*}

From Theorem \ref{BLT}, it follows that $\mu$ can be uniquely extended to a continuous linear functional $\hat{\mu} \in BV_{\frac{n}{n-1}}(\R^n)^*$ and clearly,
\begin{equation*}
 S(\hat{\mu})=\mu,
\end{equation*}
which implies that $S$ is surjective. According to Theorem \ref{BLT}, this extension preserves the operator norm and thus
\begin{equation*}
\norm{S^{-1}(\mu)}_{BV_\frac{n}{n-1}(\R^n)^*}=\norm{\hat{\mu}}_{BV_\frac{n}{n-1}(\R^n)^*}=
 \norm{\mu}_{BV_c^{\infty}(\R^n)^*} =\norm {\mu}_{\dot{W}^{1,1}(\R^n)^*},
\end{equation*}
which shows that $\mathcal{E}$ and $\mathcal{F}$ are isometrically isomorphic.
\end{proof}

\section{On an issue raised by Meyers and Ziemer}\label{MZ-ex}
 In this section, using the result of Theorem \ref{superestrella},
we construct a locally integrable function $f$ such that $f\in BV(\R^n)^{*}$ but $|f|\not \in BV(\R^n)^{*}$.  This example settles an issue raised by Meyers and Ziemer in \cite[page 1356]{MZ}. We mention that this kind of highly oscillatory function  appeared in \cite{MV} in a different context.    

\begin{proposition}\label{exam} Let $f(x)=\epsilon |x|^{-1-\epsilon}\sin(|x|^{-\epsilon}) + (n-1)|x|^{-1}\cos(|x|^{-\epsilon})$,
where $0<\epsilon<n-1$ is fixed. Then 
\begin{equation}\label{divform}
f(x) = {\rm div} \, [x|x|^{-1} \cos(|x|^{-\epsilon})].
\end{equation}
Moreover,  there exists a sequence $\{r_k\}$  decreasing to zero such that
\begin{equation}\label{Lfplus}
\int_{B_{r_k}(0)} f^+(x) dx \geq c\, r_k^{n-1-\epsilon}
\end{equation}
for a constant $c=c(n,\epsilon)>0$ independent of $k$. Here $f^+$ is the positive part of $f$. Thus by Theorem \ref{superestrella}
we see that $f$ belongs to $BV(\R^n)^{*}$, whereas $|f|$ does not. 
\end{proposition}

\begin{proof} The equality \eqref{divform} follows by a straightforward computation. To show \eqref{Lfplus}, we
let  $r_{k} =(\pi/6+2k\pi)^{\frac{-1}{\epsilon}}$ for $k=1, 2, 3, \dots$ Then we have
\begin{eqnarray*}
\int_{B_{r_{k}}(0)} f^{+}(x) dx
&=& s(n)\int_{0}^{r_{k}} t^{n} [\epsilon\, t^{-1-\epsilon}\sin(t^{-\epsilon})+ (n-1)t^{-1}\cos(t^{-\epsilon})]^{+} \frac{dt}{t}\\
&=& \frac{s(n)}{\epsilon}\int_{r_{k}^{-\epsilon}}^{\infty} x^\frac{-n}{\epsilon} [\epsilon\, x^\frac{\epsilon+1}{\epsilon} \sin(x) + (n-1) x^{\frac{1}{\epsilon}}\cos(x)]^{+} \frac{dx}{x}\\
&\geq& \frac{s(n)}{2} \sum_{i=0}^{\infty} \int_{\pi/6 +2k\pi +2i \pi}^{\pi/2+2k\pi +2i \pi} x^\frac{-n+1}{\epsilon}  dx,
\end{eqnarray*}
where $s(n)$ is the area of the unit sphere in $\R^n$. Thus using the elementary  observation
$$  \int_{\pi/2 +2k\pi+2i\pi}^{\pi/6 + 2k\pi+2(i+1)\pi} x^\frac{-n+1}{\epsilon} dx\leq  6 \int_{\pi/6 +2k\pi+2i\pi}^{\pi/2 + 2k\pi+2i\pi} x^\frac{-n+1}{\epsilon}  dx, $$
we find that
\begin{eqnarray*}
\int_{B_{r_{k}}(0)} f^{+}(x) dx&\geq& \frac{s(n)}{14}  \sum_{i=0}^{\infty} 7 \int_{\pi/6 +2k\pi +2i \pi}^{\pi/2+2k\pi +2i \pi} x^\frac{-n+1}{\epsilon} dx\\
&\geq&\frac{s(n)}{14} \sum_{i=0}^{\infty} \left(\int_{\pi/6 +2k\pi +2i \pi}^{\pi/2+2k\pi +2i \pi} x^\frac{-n+1}{\epsilon} dx+ \int_{\pi/2 +2k\pi+2i\pi}^{\pi/6 + 2k\pi+2(i+1)\pi} x^\frac{-n+1}{\epsilon} dx\right)\\
&\geq&\frac{s(n)}{14} \sum_{i=0}^{\infty} \int_{\pi/6 +2k\pi +2i \pi}^{\pi/6+2k\pi +2(i+1) \pi} x^\frac{-n+1}{\epsilon} dx\\
&=& \frac{s(n)}{14}\int_{\pi/6 +2k\pi}^{\infty} x^\frac{-n+1}{\epsilon} dx =\frac{ s(n)\, \epsilon}{14(n-1-\epsilon)} r_{k}^{n-1-\epsilon}.
\end{eqnarray*}

This completes the proof of the proposition.
\end{proof}


\section{The space $BV_0(\Omega)$}
In this section we let $\O \subset \R^n$ be a bounded open set with Lipschitz boundary. We have 
the following well known result concerning the existence of traces of functions in $BV(\O)$
(see for example \cite[Theorem 2.10]{Giusti} and \cite[Theorem 10.2.1]{ABM}):  
\begin{theorem}
Let $\O$ be a bounded open set with Lipschitz continuous boundary $\partial \O$ and let $u \in BV(\O)$. Then, there exists  a function $\varphi \in L^1(\partial \O)$ such that, for $\H^{n-1}$-almost every $x \in \partial \O$,
\begin{equation*}
\lim_{r \to 0} r^{-n}\int_{B(x,r) \cap \O}|u(y) - \varphi(x)|dy=0.
\end{equation*}
\end{theorem}
From the construction of the trace $\varphi$ (see \cite[Lemma 2.4]{Giusti}, we see that $\varphi$ is uniquely determined. Therefore, we have a well defined operator
\begin{equation*}
\gamma_0:BV(\O)\to L^1(\po\O).
\end{equation*}

We now define the space $BV_0(\O)$ as follows:
\begin{definition}
Let
$$BV_0 (\O)=\ker (\gamma_0).$$
\end{definition}

We also define another $BV$ function space with a zero boundary condition.
\begin{definition}
\label{defdos}
Let
$$\mathbb{BV}_0(\O) : =\overline{C^\infty_c(\O)},$$
where the closure is taken with respect to the intermediate convergence of $BV(\O)$.
\end{definition}

By the intermediate convergence of $BV(\O)$, we mean the following
\begin{definition}
 Let $\{u_k\}\in BV(\O)$ and $u\in BV(\O)$. We say that $u_k$ converges to $u$ in the sense of intermediate (or strict) convergence if
\begin{equation*}
u_k \to u  \tn{ strongly in } L^1(\O) \tn{ and } \int_\O|Du_k|\to \int_\O |Du|.
\end{equation*}
\end{definition}
The following theorem can be found in \cite[Theorem 10.2.2]{ABM}:

\begin{theorem}
\label{continuityoftrace}
The trace operator $\gamma_0$ is continuous from $BV(\O)$ equipped with the intermediate  convergence  onto $L^1(\po\O)$ equipped with the strong convergence. 
\end{theorem}

The following theorem is well known and can be found in many standard references including \cite{ABM, Evans, Zie89, Giusti}, but for completeness we will include the proof here.

\begin{theorem}
\label{mainapproximation}
The space $C^\infty(\O)\cap BV(\O)$ is dense in $
BV(\O)$ equipped with the intermediate convergence. Moreover, if $\O$ is a Lipschitz domain then $C^\infty(\overline{\O})$ is also dense in $BV(\O)$ for the intermediate convergence. 
\end{theorem}
\begin{proof}
We note that $C^{\infty}(\O)\cap BV(\O)=C^{\infty}(\O)\cap W^{1,1}(\O)$. For Lipschitz domains, it is proved, e.g., in \cite[page 127]{Evans} that $C^{\infty}(\overline{\O})$ is dense in $W^{1,1}(\O)$, equipped with the strong convergence. This actually holds even for  domains that possess the so-called {\it segment property} (see \cite[Theorem 3.18]{AR}).
Thus, since the strong convergence implies the intermediate convergence it follows that $C^{\infty}(\overline{\O})$ is dense in  $C^{\infty}(\O)\cap BV(\O)$ in the intermediate convergence. Therefore, if $C^{\infty}(\O)\cap BV(\O)$ is dense in $BV(\O)$ for the intermediate convergence, the second statement of the theorem holds.
Let $\varepsilon>0$ and $u\in BV(\O)$. We decompose $\Omega$ as follows: 
\begin{equation*}
\O=\bigcup^\infty_{i=0}\O_i, \int_{\O\setminus\O_0}|Du|<\varepsilon \tn{ and } \O_i\subset\subset\O_{i+1}.
\end{equation*}
We consider the open cover $\{C_i\}$ defined as follows:
\begin{eqnarray*}
C_1&:=&\O_2\\
C_i &:=&\O_{i+1}\setminus {\overline{\O}}_{i-1},\quad  i\ge 2.
\end{eqnarray*}
Let $\{\varphi_i\}$ be a partition of unity associated to $\{C_i\}$; that is,
\[\varphi_i\in C^\infty_c(C_i),\quad 0\le\varphi_i\le 1,\quad \sum^\infty_{i=1}\varphi_i=1.\]
Note that $\varphi_1\equiv 1$ on $\O_1$. Let $\rho$ be a standard mollifier as in the proof of Theorem \ref{pmcondicion}. For each $i$, choose $\varepsilon_i>0$ so that:
\begin{eqnarray}
&&\tn{ spt } (\rho_{\varepsilon_i} * \varphi_i u)\subset C_i, \nonumber\\
&&\int_\O |\rho_{\var_i} * (u\varphi_i)-u\varphi_i|<{\var\over 2^i}\label{letraA},
\\&&\int_\O |\rho_{\var_i} * (u D\varphi_i)-u D\varphi_i|<{\var\over 2^i}\label{letraB},\\
&&\bigg|\int_\O |\rho_{\var_1} * (\varphi_1Du)| dx-\int_\O|\varphi_1 Du|\bigg|<\var.
\label{letraC}
\end{eqnarray}
Define \[ u_\var : =\sum^\infty_{i=1}\, \rho_{\var_i} * (u\varphi_i).\]
Then
\[\int_\O|u-u_\var|\, dx\le\sum^\infty_{i=1}\int_\O\, |\rho_{\var_i} *(u\varphi_i)-u\varphi_i|\, dx <\var, \quad \rm{by\ \eqref{letraA}}.\]
We have
\begin{eqnarray*}
Du_\var&=& \sum^\infty_{i=1}\, \rho_{\var_i} * (\varphi_i Du)+\sum^\infty_{i=1}\, \rho_{\var_i} * (u D \varphi_i)\\
&=& \sum^\infty_{i=1}\, \rho_{\var_i} * (\varphi_i Du)+\sum^\infty_{i=1}\, \big(\rho_{\var_i} * (u D \varphi_i)-uD\varphi_i\big).
\end{eqnarray*}
Then, on the one hand,
\begin{eqnarray}
\bigg|\int_\O|\rho_{\var_1} * (\varphi_1 Du)|\, dx -\int_\O|Du_\var|\bigg| &\le& 
\sum^\infty_{i=2} \int_\O|\rho_{\var_i} * (\varphi_i Du)|\, dx+\sum^\infty_{i=1}\int_\O|\rho_ {\var_i} * (u D\varphi_i)-u D\varphi_i|\, dx \nonumber\\
 &\le& \sum^\infty_{i=2} \int_\O|\rho_{\var_i} * (\varphi_i Du)| +\var, \quad\rm{by}\ \eqref{letraB} \nonumber\\
&\le& \sum^\infty_{i=2} \int_\O |\varphi_i Du|+\var, \tn{ by a property of convolution} \nonumber\\
&\le& \int_{\O\setminus\O_0}|Du|+\var \nonumber\\
&\le&  \var+\var=2\var \label{one}.
\end{eqnarray}

On the other hand,
\begin{eqnarray}
\bigg|\int_\O|\rho_{\var_1} * (\varphi_1 Du)|\, dx -\int_\O|Du|\bigg| \nonumber
&=& \bigg|\int_\O|\rho_{\var_1} * (\varphi_1 Du)|\, dx -\int_\O |\varphi_1 Du|-\int_\O(1-\varphi_1)|Du|\bigg| \nonumber\\
&\le & \var +\int_\O(1-\varphi_1)|Du|, \tn{ by } \eqref{letraC} \nonumber\\
&\le & \var +\int_{\O\setminus\O_0}|Du|\le 2\var, \tn{ since } \varphi_1 \equiv 1 \tn{ on } \O_1 \label{two}.
\end{eqnarray}

From \eqref{one} and \eqref{two}:
$$\bigg|\int_\O|Du_\var|-\int_\O|Du|\bigg|<4\var.$$
\end{proof}

\begin{theorem}
\label{muyimportante}
Let $\Omega$ be any bounded open set with Lipschitz boundary. Then $BV_c(\Omega)$ is dense in $BV_0(\Omega)$ in the strong topology of $BV(\Omega)$.
\end{theorem}

\begin{proof}
We consider first the case $u\in BV_0(C_{R,T})$, where $C_{R,T}$ is the
open cylinder 
\begin{equation*}
C_{R,T}= \mathcal{B}_{R} \times (0,T),
\end{equation*}
$\mathcal{B}_{R}$ is an open ball of radius $R$ in $\R^{n-1}$, and $\tn{supp}(u) \cap \partial C_{R,T}=\tn{supp}(u) \cap (\mathcal{B}_R\times\{0\})$. A generic point in $C_{R,T}$ will be denoted by $(x',t)$, with $x' \in \mathcal{B}_R$ and $t\in (0,T)$.
 From Theorem \ref{mainapproximation}, we can approximate $u$ with a sequence of functions
$u_k\in C^\infty(\overline{C_{R,T}})$  
such that 
\begin{eqnarray} \label{three}
u_k \to u \tn{ in }\ L^1(C_{R,T}) \tn{ and } \int_{C_{R,T}} |Du_k|\to \int_{C_{R,T}} |Du|.
\end{eqnarray}
Notice that the condition $\tn{supp}(u) \cap \partial C_{R,T}=\tn{supp}(u) \cap (\mathcal{B}_R\times\{0\})$ implies that 
$$\gamma_0(u_k)\rtangle (\partial C_{R,T}\setminus (\mathcal{B}_R\times\{0\})) \equiv 0.$$ From Theorem \ref{continuityoftrace}, $\gamma_0(u_k)\to \gamma_0(u) \textnormal{ in } L^1(\partial C_{R,T})$ and hence
\begin{equation}\label{four}
 \gamma_0(u_k) \rtangle (\mathcal{B}_R \times\{0\})=u_k \big|_{(\mathcal{B}_R \times\{0\})}\to 0 \tn{ in } L^1(\mathcal{B}_R \times\{0\}).
\end{equation}
For $x'\in \mathcal{B}_R$, $0 \leq x_n \leq T$, we have
\begin{equation*}
u_k(x',x_n)-u_k(x',0)=\int^{x_n}_0{\partial u_k\over \partial x_n}\, (x', t) dt,
\end{equation*}
and hence,
\begin{equation}
\label{help}
|u_k(x',x_n)|\le |u_k(x',0)|+\int^{x_n}_0 \left| {\partial u_k\over \partial x_n}\, (x', t)\right| dt.
\end{equation}
We integrate both sides in \eqref{help} to obtain:
\begin{equation}
\label{importante}
\int_{\mathcal{B}_R}|u_k(x', x_n)| dx'\le \int_{\mathcal{B}_R}|u_k(x', 0)| dx'+
\int^{x_n}_0\int_{\mathcal{B}_R}|Du_k(x', t)| dx'\, dt.
\end{equation}
From \eqref{four} we have
\begin{equation}
\label{ya}
 \lim_{k \to \infty} \int_{\mathcal{B}_R}|u_k(x',0)| dx' = 0,
\end{equation}
and thus, letting $k\to\infty$ in \eqref{importante} and using \eqref{ya}, \eqref{three} and Lemma \ref{muyimportante-1} (in particular, \eqref{clave} with $A:= \mathcal{B}_R \times (0,x_n)$ for a.e. $0< x_n < T$) we obtain
\begin{equation}
\label{mostimportant}
\int_{\mathcal{B}_R}|u(x', x_n)| dx'\le \int^{x_n}_0\int_{\mathcal{B}_R}|Du|  = \norm{Du}(\mathcal{B}_R \times(0,x_n))  \tn{  for a.e.  } 0< x_n < T.
\end{equation}
Consider a  function $\varphi\in C_c^\infty(\R)$ such that $\varphi$ is decreasing in $[0, +\infty)$ and satisfies 
\begin{equation*}
\varphi\equiv 1 \tn{ on } [0,1],\varphi\equiv 0 \tn{ on } \R\setminus [-1,2], 0\le\varphi\le 1.
\end{equation*}
We define
\begin{eqnarray}
&\varphi_k(t)=\varphi(k t)\nonumber, \quad k=1,2, \dots\\
 &v_k(x',t) = (1-\varphi_k(t))u(x',t)\label{six}.
\end{eqnarray}
Clearly, $v_k\to u$ in $L^1(C_{R,T})$. Also, if $u\geq 0$ then $v_k \uparrow u$ since $\varphi$ is decreasing in $[0, +\infty)$.  Moreover, 
\begin{eqnarray*}
{\partial v_k\over \partial t}&=&(1-\varphi_k){\partial u\over\partial t}-k\varphi'(k t)u,\\
D_{x'}v_k&=&(1-\varphi_k)D_{x'}u.
\end{eqnarray*}
Thus we have 
\begin{eqnarray*}
\int_{C_{R,T}}|D v_k-Du|&=&\int_{C_{R,T}}\bigg|\Big(D_{x'}u-\varphi_k D_{x'}u,{\partial u\over\partial t}-\varphi_k{\partial u\over\partial t}-k\varphi'(k t)u \Big)- \Big(D_{x'}u, {\partial u\over\partial t}\Big)\bigg|\\
&=&\int_{C_{R,T}}\bigg| \Big(-\varphi_k D_{x'}u, -\varphi_k{\partial u\over\partial t}-k\varphi'(k t)u \Big )\bigg|.
\end{eqnarray*}
Since $\varphi_k(t) =0$ for $ t >{2\over k}$ we have the following:
\begin{eqnarray}
\int_{C_{R,T}}|D v_k-Du|&\le& C\bigg(\int_{C_{R,T}}\varphi_k|Du|+\int_{C_{R,T}}k|\varphi'(kx_n)||u|\bigg)\nonumber\\
&\le& C\int_0^{2/k}\int_{\mathcal{B}_R}|Du|+C\, k\int_0^{2/k}\int_{\mathcal{B}_R}|u(x',t)| dx'dt\nonumber\\
&\le& C\int_0^{2/k}\int_{\mathcal{B}_R}|Du|+C\, k\int_0^{2/k}\norm{Du}(\mathcal{B}_R \times(0,t)) dt, \tn{ by } \eqref{mostimportant}\nonumber\\
&\leq&C\int_0^{2/k}\int_{\mathcal{B}_R}|Du|+C\, k \cdot \norm{Du}(\mathcal{B}_R \times(0,2/k))\cdot \int^{2/k}_0 dt\nonumber\\
&=&C\int_0^{2/k}\int_{\mathcal{B}_R}|Du|+C\, k\cdot\frac 2k\cdot  \int_0^{2/k}\int_{\mathcal{B}_R}|Du|\nonumber\\
&=&C\int_0^{2/k}\int_{\mathcal{B}_R}|Du|.\label{oneone}
\end{eqnarray}
Since $\norm{Du}$ is a Radon measure and $\cap_{k=1}^{\infty}(\mathcal{B}_R\times (0,\frac{2}{k}))= \emptyset$ we find
\begin{equation*}
\int_0^{2/k}\int_{\mathcal{B}_R}|Du| \to 0, \tn{ as } k \to \infty,
\end{equation*}
which by \eqref{oneone} yields
\begin{equation*}
\lim_{k\to\infty}\int_{C_{R,T}}|Dv_k-Du|=0.
\end{equation*}
Thus
\begin{equation}
\label{densidadenelcilindro}
v_k\to u \tn{ in the strong topology of } BV(C_{R,T}).
\end{equation}


We consider now the general case of a  bounded  open set $\O$ with Lipschitz boundary and let $u \in BV_0(\Omega)$. For each point $x_0 \in \partial \Omega$, there exists a neighborhood $A$ and a bi-Lipschitz function $g:B(0,1)\to A$ that maps $B(0,1)^{+}$ onto $A\cap \Omega$ and the flat part of $\partial B(0,1)^{+}$ onto $A \cap \partial\O$. A finite number of such sets $A_1, A_2, \dots, A_n$ cover $\partial\Omega$. By adding possibly an additional open set $A_0 \subset \subset \Omega$, we get a finite covering of $\overline{\Omega}$. Let $\{\alpha_i\}$ be a partition of unity relative to that covering, and let $g_i$ be the bi-Lipschitz map relative to the set $A_i$ for $i=1, 2, \dots,N$. For each $i \in \{1,2, \dots,N\}$ the function 
$$U_i=(\alpha_i u) \circ g_i$$
belongs to $BV_0(B(0,1)^{+})$, and has support non-intersecting the curved part of $\partial B(0,1)^{+}$. Thus, we can extend $U_i$ to the whole cylinder $C_{1,1}:=\mathcal{B}_{1}(0) \times (0,1)$ by setting $U_i$ equal to zero outside $B(0,1)^{+}$. By \eqref{densidadenelcilindro}, for each $\varepsilon >0$, we can find a function $W_i \in BV_c(C_{1,1})$ such that 
\begin{equation}
\label{callie}
\norm{W_i-U_i}_{BV(C_{1,1})} \leq \varepsilon,
\end{equation}
for $i=1, 2, \dots,N$. Letting now
\begin{equation*}
 w_i=W_i \circ g_i^{-1}, \,\, i=1,2, \dots,N, 
\end{equation*} 
we have $w_i \in BV_c(A_i \cap \Omega)$ and 
\begin{eqnarray}
 \norm{D(w_i-\alpha_i u)}(A_i \cap \O)&=& \norm{ D(W_i \circ g_i^{-1} -((\alpha_i u) \circ g_i) \circ g_i^{-1}) } (A_i \cap \O) \nonumber\\
 &=&\norm{D(g_{i\#}(W_i-(\alpha_i u)\circ g_i))} (A_i \cap \Omega)\nonumber\\
 &\leq& C\,  g_{i\#} \norm{D(W_i-(\alpha_i u)\circ g_i)} (A_i \cap \Omega), \textnormal{ by  \cite[Theorem 3.16]{afp}} \nonumber\\
 &=&C \int_{g_i^{-1} (A_i \cap \Omega)} |D(W_i-U_i)|, \textnormal{by definition of $g_{i\#}$ acting on measures}\nonumber\\
 &=& C \int_{B(0,1)^{+}} |D(W_i-U_i)|\nonumber\\
 &\leq& C\, \varepsilon, \textnormal{ by } \eqref{callie}.
\label{DWiai}
\end{eqnarray}
Here $C=\max_{i}\{[\Lip(g_i)]^{n-1}\}$ (see \cite[Theorem 3.16]{afp}). Let $w_0=\alpha_0u$. Then $w_0 \in BV_c(\O)$. Define
\begin{equation*}
w= \sum_{i=0}^{N}w_i.
\end{equation*}
We have $w \in BV_c(\Omega)$, and by \eqref{DWiai}
\begin{eqnarray*}
\norm{D(w-u)}(\Omega) &\leq& \sum_{i=0}^N \norm{D(w_i-\alpha_i u)}(A_i\cap \Omega)\nonumber\\
&=& \sum_{i=1}^N \norm{D(w_i-\alpha_i u)}(A_i\cap \Omega) \nonumber\\
&\leq& N C\, \varepsilon.
\end{eqnarray*}
Likewise, by \eqref{callie} and a change of variables we have 
\begin{equation*}
\norm{w-u}_{L^1(\Omega)} \leq \sum_{i=0}^N \norm{w_i-\alpha_i u}_{L^1(A_i\cap \Omega)}\leq  \sum_{i=1}^N \norm{w_i-\alpha_i u}_{L^1(A_i\cap \Omega)} \leq N c\, \varepsilon.
\end{equation*}
 Thus $\overline{BV_c(\Omega)}=BV_0(\Omega)$ in the strong topology of $BV(\O)$. 
\end{proof}

\begin{remark}\label{INC-apr}
By \eqref{six} and the construction of $w$ in the proof of Theorem \ref{muyimportante} above, we see that each $u\in BV_0(\O)$ can be approximated 
by a sequence $\{u_k\}\subset BV_c(\O)$ such that  $u_k=u$ in $\O\setminus N_k$ for a set $N_k=\{x\in \Omega: d(x, \partial\O)\leq \delta(k)\}$ with 
$\delta(k)\rightarrow 0$ as $k\rightarrow +\infty$. Moreover, if $u \geq 0$ then so is  $u_k$  and $u_k \uparrow u$ as $k$ increases to $+\infty$.

\end{remark}

We will also need the following density result.

\begin{lemma}\label{density2}
$BV_0^\infty(\O)$ is dense in $BV_0(\O)$. Likewise, $BV_c^\infty(\O)$ is dense in $BV_c(\O)$, and $BV^\infty(\O)$ is dense in $BV(\O)$ in the strong topology of $BV(\O)$.
\end{lemma}
\begin{proof} We shall only prove the first statement as the others  can be shown in a similar way.
Let $u\in BV_0^+(\Omega)$ and define
\begin{equation*}
u_j:=u\wedge j, \quad j=1,2, \dots
\end{equation*}
Obviously, $u_j \to u$ in $L^1(\O)$. We will now show that $\norm{D(u-u_j)}(\O)\to 0$. The coarea formula yields
\begin{eqnarray*}
\int_{\O}|D(u-u_j)|&=&\int_0^\infty \mathcal{H}^{n-1}(\O \cap \partial^*\{u-u_j>t\})dt\\
&=&\int_0^\infty \mathcal{H}^{n-1}(\O \cap \partial^*\{u-j>t\})dt\\
&=&\int_0^\infty \mathcal{H}^{n-1}(\O \cap \partial^*\{u>j+t\})dt\\
&=&\int_j^\infty \mathcal{H}^{n-1}(\O \cap \partial^*\{u>s\})ds.
\end{eqnarray*}
Since $\int_0^\infty\mathcal{H}^{n-1}(\O \cap \partial^*\{u>s\})ds < \infty$, the Lebesgue dominated convergence theorem implies that
\begin{equation}
\label{lunes}
\int_{\O}|D(u-u_j)|\to 0 \textnormal{ as } j \to \infty.
\end{equation}
If $u\in BV_0(\O)$, we write $u=u^+-u^-$ and define $f_j=u^+\wedge j$ and $g_j=u^-\wedge j$. Thus $f_j-g_j\in BV_0(\O)$ and 
\begin{eqnarray*}
\int_{\O}|D(u-(f_j-g_j))|&=&\int_{\O}|D u^+-D u^--D f_j + D g_j|\\
&\leq& \int_{\O}|D(u^+-f_j)|+\int_{\O}|D(u^--g_j)|\\
&\to& 0 \textnormal{ as } j \to \infty,
\end{eqnarray*}
due to \eqref{lunes}. That completes the proof of the lemma.
\end{proof}

We are now ready to prove the main theorem of this section that makes precise the definition of the space of functions of bounded variation in $\O$ with zero trace on the boundary of $\O$.

\begin{theorem} \label{sonlomismo}
 $\mathbb{BV}_0(\O)=BV_0(\O)$.
\end{theorem}
\begin{proof}
 Let 
 Let $u\in \mathbb{BV}_0(\O)$. Then Definition \ref{defdos} implies the existence of a sequence $\{u_k\} \in C^\infty_c(\O)$ such that
\begin{equation*}
u_k\to u \quad \tn{in }\ L^1(\O) \tn{  and } \int_\O|Du_k|\to \int_\O|Du|.
\end{equation*}
Since $u_k\in C^\infty_c(\O)$, we have  $\gamma_0(u_k)\equiv 0$. Then Theorem \ref{continuityoftrace} yields
\begin{equation*}
 \gamma_0(u_k)\to \gamma(u) \quad {\rm  in~} L^1(\partial\Omega),
 \end{equation*}
 and so
 \begin{equation*}
 \gamma(u)=0 \quad \tn{ and ~ } u\in BV_0(\O).
 \end{equation*}
 In the other direction, let $u \in BV_0(\O)$. Then, from Theorem \ref{muyimportante} there exists a sequence $u_k \in BV_c(\O)$ such that
\begin{equation}
\label{papitas1}
\lim_{k \to \infty} \int_{\O}|u_k -u|=
 \lim_{k \to \infty} \int_{\O}|Du_k - Du|=0.
 \end{equation}
 Given a sequence $\varepsilon_k \to 0$, we consider the sequence of mollifications
\begin{equation*}
w_k:= u_k * \rho_{\varepsilon_k}.
\end{equation*}
We can choose $\varepsilon_k$ sufficiently small to have
\begin{equation*}
w_k \in C_c^{\infty}(\O).
\end{equation*}
Also, for each $k$,
\begin{equation*}
\lim_{\varepsilon \to 0} \int_{\O} |D(u_k*\rho_{\varepsilon})| = \int_{\O}|Du_k|,
\end{equation*}
and
\begin{equation*}
\lim_{\varepsilon \to 0} \int_{\O} |u_k*\rho_{\varepsilon}-u_k| =0.
\end{equation*}
Thus we can choose $\varepsilon_k$ small enough so that, for each $k$,
\begin{equation}
\label{papitas2}
  \left| \int_{\O} |D(u_k*\rho_{\varepsilon_k})| - \int_{\O}|Du_k|\right |\leq \frac{1}{k},
 \end{equation}
and 
\begin{equation}
\label{papitas22}
 \int_{\O} |u_k*\rho_{\varepsilon_k}-u_k| \leq \frac{1}{k}.  
\end{equation}
Using \eqref{papitas22} and \eqref{papitas1} we obtain
\begin{equation}
\label{yaaa}
\lim_{k \to \infty} \int_{\O}|w_k -u| \leq \lim_{k \to \infty} \int_{\O}|w_k -u_k| + \lim_{k \to \infty} \int_{\O}|u_k -u|=0.
\end{equation}
Also, letting $k \to \infty$ in \eqref{papitas2} and using \eqref{papitas1}, we obtain
\begin{equation}
\label{yaaaa}
\lim_{k \to \infty} \int_{\O}|D(u_k* \rho_{\varepsilon_k})| = \int_{\O}|Du|.
\end{equation}
From \eqref{yaaa} and \eqref{yaaaa} we conclude that $w_k \to u$ in the intermediate convergence which implies that $u\in \mathbb{BV}_0(\O)$.
\end{proof}

With Theorem \ref{sonlomismo} we can now prove the following Sobolev's inequality for functions in $BV_0(\O)$:

\begin{theorem}
\label{sobolevbv0}
Let $u \in BV_0(\O)$, where $\O$ is a bounded open set with Lipschitz boundary. Then
\begin{equation*}
\norm{u}_{L^{\frac{n}{n-1}}(\O)} \leq C \norm{D u}(\O),
\end{equation*}
for a constant $C=C(n)$.
\end{theorem}
\begin{proof}
The Sobolev inequality  for smooth functions states that
\begin{equation}
\label{arriba1}
\norm{u}_{L^{\frac{n}{n-1}}(\R^n)}\leq C \int_{\R^n}|D u| \tn{  for each } u \in C_c^\infty(\R^n).
\end{equation}
From Theorem \ref{sonlomismo} there exists a sequence $u_k \in C_c^{\infty}(\O)$ such that
\begin{equation}
\label{pompo1}
u_k\to u \quad \tn{in }\ L^1(\O) \tn{  and } \int_\O|Du_k|\to \int_\O|Du|.
\end{equation}
Since $u_k \to u$ in $L^1(\O)$ then there exists a subsequence $\{u_{k_j}\}$ of $\{u_k\}$ such that
\begin{equation*}
u_{k_j} (x) \to u(x) \tn{ for a.e. } x \in \O.
\end{equation*}
Using Fatou's Lemma and \eqref{arriba1}, we obtain
\begin{equation}
\label{gracias}
\int_{\O}|u|^{\frac{n}{n-1}} \leq  \liminf_{j \to \infty} \int_{\O} |u_{k_j}|^{\frac{n}{n-1}}   \leq \liminf_{j \to \infty} 
\left(C \int_{\O} |D u_{k_j}| \right)^{\frac{n}{n-1}}.
\end{equation}
Finally,  using \eqref{pompo1} in \eqref{gracias} we conclude
\begin{equation*}
\left(\int_{\O}|u|^{\frac{n}{n-1}}\right)^{\frac{n-1}{n}}  \leq C\int_{\O}|Du|.
\end{equation*} 

\end{proof}

By Theorem \ref{sobolevbv0}, we see that $\norm{u}_{BV(\O)}$ is equivalent to $\norm{Du}(\O)$ whenever $u\in BV_0(\O)$ (or $\mathbb{BV}_0(\O)$) and $\O$ is a bounded Lipschitz domain. Thus, for the rest of the paper we will equip $BV_0(\O)$ with the homogeneous norm:
$$\norm{u}_{BV_0(\O)}=\norm{Du}(\O).$$ 

From Theorem \ref{muyimportante} and Lemma \ref{density2} we obtain

\begin{corollary}
\label{density3}
Let $\O$ be any bounded open set with Lipschitz boundary. Then $BV_c^\infty(\O)$ is dense in $BV_0(\O)$.
\end{corollary}

\section{Characterizations of  measures in $BV_0(\Omega)^*$}
 First, as in the case of $\R^n$, we make precise the definitions of  measures in the spaces $W_0^{1,1}(\O)^*$ and $BV_{0}(\O)^*$.
\begin{definition}
 For a bounded open set $\O$ with Lipschitz boundary, we let
\begin{equation*}
\mathcal{M}_{loc}(\O) \cap W_0^{1,1}(\O)^*:= \{T \in W_0^{1,1}(\O)^* : T(\varphi)=\int_{\O} \varphi d\mu \text{ for some } \mu \in \mathcal{M}_{loc}(\O), \forall \varphi \in C_c^{\infty}(\O)\}. 
\end{equation*}
Therefore, if $\mu \in \mathcal{M}_{loc}(\O) \cap W_0^{1,1}(\O)^*$, then the action $<\mu,u>$ can be uniquely defined for all $u \in W_0^{1,1}(\O)$ (because of the density of $C_c^{\infty}(\O)$ in $W_0^{1,1}(\O))$. 
\end{definition}

\begin{definition}
For a bounded open set $\O$ with Lipschitz boundary, we let
\begin{equation*}
\mathcal{M}_{loc}(\O) \cap BV_{0}(\O)^*:=\{T \in BV_{0}(\O)^*: T(\varphi)=\int_{\O} \varphi^* d\mu \text{ for some } \mu \in \mathcal{M}_{loc}(\O), \forall \varphi \in BV_c^{\infty}(\O)\},
\end{equation*} 
where $\varphi^*$ is the precise representative of $\varphi$. Thus, if $\mu \in \mathcal{M}_{loc}(\O) \cap BV_{0}(\O)^*$, then the action $<\mu,u>$ can be uniquely defined for all $u \in BV_{0}(\O)$ (because of the density of $BV_c^{\infty}(\O)$
in $BV_{0}(\O)$ by Corollary \ref{density3}).
\end{definition}

We will use the following characterization of  $W_{0}^{1,1}(\O)^*$ whose proof is completely analogous to that of  Lemma \ref{lemma1}.
\begin{lemma}
\label{lemma2}
Let $\Omega$ be any bounded open set with Lipschitz boundary. The distribution $T$ belongs to $W_{0}^{1,1}(\O)^*$ if and only if $T=\div \FF$ for some vector field $\FF \in L^{\infty}(\O,\R^n)$. Moreover,
\begin{equation*}
 \norm{T}_{W_0^{1,1}(\O)^*}= \min \{ \norm{\FF}_{L^{\infty}(\O,\R^n)} \},
\end{equation*}
where the minimum is taken over all $\FF \in L^{\infty}(\O,\R^n)$ such that $\div \FF =T$. Here we use the norm
$$\norm{\FF}_{L^{\infty}(\O,\R^n)}:=  \norm{(F_1^2 + F_2^2 + \cdots + F_n^2)^{1/2}}_{L^{\infty}(\O)} \text{ for } \FF=(F_1, \dots, F_n).$$
\end{lemma}

We are now ready to state the main result of this section.
\begin{theorem}
\label{superestrellabounded}
Let $\Omega$ be any bounded open set with Lipschitz boundary  and $\mu\in \mathcal{M}_{loc}(\Omega)$. Then, the following are equivalent:

{\bf (i)} There exists a vector field $\FF\in L^\infty (\O, \R^n)$ such that $\div \FF=\mu$.

{\bf (ii)} $|\mu(U)| \leq C\, \mathcal{H}^{n-1}(\partial U)$ for any smooth open (or closed) set $U \subset  \subset \Omega$ with $\mathcal{H}^{n-1}(\partial U) < +\infty$.

 {\rm {\bf(iii)}} $\H^{n-1}(A)=0$ implies $\norm{\mu}(A)=0$ for all Borel sets $A  \subset \Omega$ and there is a constant $C$ such that, for all $ u \in BV_c^\infty(\O)$,
 $$
|<\mu,u>|:=\left|\int_{\O} u^*d\mu\right| \leq C\int_{\O}| D u|, 
$$
where $u^*$ is the representative in the class of $u$ that is defined
$\H^{n-1}$-almost everywhere.

{\rm {\bf(iv)}}  $\mu \in BV_0(\O)^*$. The action of $\mu$ on any $u \in BV_0(\O)$
is defined (uniquely) as
$$ 
<\mu, u>:= \lim_{k \to \infty} <\mu,u_k>= \lim_{k \to \infty} \int_{\O} u_{k}^*d\mu,
$$
where $u_k \in BV_c^\infty(\O)$ 
converges to $u$ in $BV_0(\O)$. In particular, if $u\in BV_c^\infty(\O)$ then
$$<\mu, u>=\int_\O u^* d\mu,$$
and moreover, if $\mu$ is a non-negative measure then,
for all $u \in BV_0(\O)$,
$$
<\mu,u>=\int_{\O} u^*d\mu.
$$
\end{theorem}

\begin{proof}
Suppose {\bf(i)} holds. Then for every $\varphi \in C_c^{\infty}(\Omega)$ we have
 \begin{equation*}
 \int_{\O} \FF \cdot D \varphi dx = -\int_{\O} \varphi d\mu.
 \end{equation*}
 Let $U \subset \subset \Omega$ be any open (or closed) set with smooth boundary satisfying $\H^{n-1}(\partial U) < \infty$. We proceed as in Theorem \ref{superestrella} and consider the characteristic function $\chi_U$ and the sequence $u_k:= \chi_U * \rho_{1/k}$. Since $U$ is strictly contained in $\O$, for $k$ large enough, the support of $\{u_k\}$ are contained in $\O$. We can then proceed exactly as in Theorem \ref{superestrella} to conclude that 
\begin{equation*} 
|\mu(U)|\le C\, {\mathcal H}^{n-1}(\partial U),
\end{equation*}
where $C=\norm{\FF}_{L^\infty(\O)}$ for closed sets $U$ and  $C=3\norm{\FF}_{L^\infty(\O)}$ for open sets $U$.

If $\mu$ satisfies {\bf (ii)} with a constant $C>0$, then Corollary \ref{pmcondition2} implies that $\norm{\mu}< <\H^{n-1}$. 
 We let $u \in BV_c^\infty(\O)$ and $\{\rho_\ve\}$ be a
standard sequence of mollifiers. Consider  the convolution $\rho_\ve * u$ and note that $\rho_\ve * u \in C_c^\infty(\O)$, for $\varepsilon$ small enough.
Then as in the proof of Theorem \ref{superestrella} we have, for $\varepsilon$ small enough,
\begin{equation*} 
\left|\int_{\O}\rho_\ve*u d\mu\right| \leq 
 C \int_{\O}|D u|.
\end{equation*}
Sending $\ve$ to zero and using the dominated convergence theorem yield
\begin{equation*}
\left|\int_{\O}u^* d\mu\right| \leq C\int_{\O}|D u|,
\end{equation*}
with the same constant $C$ as in {\bf (ii)}. This gives  {\bf (ii)} $\Rightarrow$ {\bf (iii)}.

 From {\bf(iii)} we obtain that the linear operator
\begin{equation}\label{rep-for-b}
 T(u):= <\mu,u>=\int_{\O}u^* d\mu, \quad u \in BV_c^\infty(\O)
\end{equation}
is continuous and hence it can be uniquely extended, since 
$BV_c^\infty(\O)$ is dense in $BV_0(\O)$ (Corollary \ref{density3}), 
to the space $BV_0(\O)$.

Assume now that $\mu$ is non-negative. We take
$u \in BV_0(\O)$ and consider the positive and negative parts $(u^*)^+$
and $(u^*)^-$ of the representative $u^*$. By Remark \ref{INC-apr}, there is an increasing sequence of nonnegative functions $\{v_k\}\subset
BV_c(\O)$  that converges to $(u^{*})^{+}$  pointwise  and in the $BV_0$ norm. Therefore, using \eqref{rep-for-b} we have
$$T(v_k \wedge j) =\int_\O v_k \wedge j d\mu, \quad j=1,2, \dots$$
We first send $j$ to infinity and then $k$ to infinity. Using the continuity of $T$, \eqref{lunes},  and the monotone convergence theorem we get
$$
T((u^*)^+)= \int_{\O}(u^*)^+ d\mu.
$$
We proceed in the same way for  $(u^*)^-$ and thus by linearity we conclude
$$
T(u)=T((u^*)^+)-T((u^*)^-)=\int_{\O}(u^*)^+ -(u^*)^- d\mu=\int_{\O}u^* d\mu.
$$

Finally, to prove that {\bf (iv)} implies {\bf (i)} we take $\mu \in BV_0(\O)^*$. Since $W_0^{1,1}(\O)\subset BV_0(\Omega)$ then
\begin{equation*}
\tilde{\mu}:= \mu \rtangle W_0^{1,1}(\O) \in W_0^{1,1}(\O)^*, 
\end{equation*}        
and therefore Lemma \ref{lemma2} implies that there exists $\FF \in L^\infty(\O,\R^n)$ such that $\div \FF =\tilde{\mu}$ and thus, since $C_c^{\infty} \subset W_0^{1,1}(\O)$, we conclude that $\div \FF= \mu$  in the sense of distributions.
\end{proof}

\begin{remark} If $\O$ is  a bounded domain containing the origin  then the function $f$  given in Proposition \ref{exam} belongs to $BV_0(\O)^{*}$
but $|f|$ does not. 
\end{remark}

Theorem \ref{superestrellabounded} and Lemma \ref{lemma2} immediately imply the following new result which states that the set of measures in $BV_0(\O)^*$ coincides with that of $W^{1,1}_0(\O)^*$. 

\begin{theorem}
\label{nuevo1}
The normed spaces  $\M_{loc}(\O) \cap BV_0(\O)^*$ and $\M_{loc}(\O) \cap W^{1,1}_0(\O)^*$ are isometrically isomorphic.
\end{theorem}

The proof of Theorem \ref{nuevo1} is similar to that of Theorem \ref{nuevo} but this time one uses Theorem \ref{superestrellabounded}  and Corollary \ref{density3} in place of Theorem \ref{superestrella} and Theorem \ref{density}, respectively. Thus we shall omit its proof.


\section{Finite measures in $BV(\Omega)^*$}
 In this section, we  characterize all {\it finite} signed measures  that belong to $BV(\O)^*$. Note that the finiteness condition here is necessary at least 
for {\it positive} measures in $BV(\O)^*$. By a measure $\mu\in BV(\O)^*$ we mean that the inequality 
\begin{equation*}
\left | \int_\Omega u^{*} d\mu \right | \leq C \|u\|_{BV(\Omega)}
\end{equation*}
holds for all $u\in BV^\infty(\O)$. By Lemma \ref{density2} we see that such a $\mu$ can be uniquely extended to be a continuous linear  
functional in $BV(\O)$.

We will use the following result, whose proof can be found in \cite[Lemma 5.10.14]{Zie89}:
\begin{lemma}\label{ellema}
Let $\Omega$ be an open set with Lipschitz boundary and $u\in BV(\Omega)$.
Then, the extension of $u$ to $\Bbb R^n$ defined by
\begin{equation*}
u_0(x)=\begin{cases}u(x),&x\in\Omega\\
0,&x\in\Bbb R^n\backslash \Omega \end{cases}
\end{equation*}
satisfies that $u_0\in BV(\Bbb R^n)$ and
\begin{equation*}
\|u_0\|_{BV(\Bbb R^n)}\leq C \|u\|_{BV(\Omega)},
\end{equation*}
where $C=C(\Omega)$.
\end{lemma}

\begin{theorem}
\label{resultforbounded}
Let $\Omega$ be an open set with Lipschitz boundary and let $\mu$ be a finite signed measure in $\O$. Extend $\mu$ by zero to $\R^n \setminus \O$ by setting $\mu(\R^n \setminus \O)=0$.
 Then, $\mu\in BV(\Omega)^*$ if and only if
\begin{equation}\label{condicion}
|\mu (U)|\leq C\, \H^{n-1} (\partial U)
\end{equation}
for every smooth open set $U\subset\Bbb R^n$ and a constant $C=C(\Omega, \mu)$.
\end{theorem}

\begin{proof}Suppose that $\mu\in BV(\Omega)^*$.
Let $u\in BV_c^{\infty}(\Bbb R^n)$ and assume that $u$ is the representative that is defined $\H^{n-1}$-almost everywhere. Consider $v:=u\chi_\Omega$ and note that  $v \rtangle \O \in BV^\infty(\Omega)$ since $Dv$ is a finite vector-measure in $\R^n$  given by 
\begin{equation*}
Dv=u D\chi_\Omega+\chi_\Omega Du,
\end{equation*}
and therefore,
\begin{eqnarray}
\int_\Omega |Dv|&=&\int_\Omega |u D\chi_\Omega+\chi_\Omega Du|\leq\int_\Omega |u| |D\chi_\Omega| + \int_\Omega |Du|\nonumber\\
&=&\int_\Omega |Du|\leq\int_{\Bbb R^n} |Du| = \|u\|_{BV(\Bbb R^n)} < +\infty.\label{omegita2}
\end{eqnarray}
Since $\mu\in BV(\Omega)^*$ there exists a constant $C=C(\Omega,\mu)$ such that
\begin{equation}\label{omegita1}
\bigg|\int_\Omega v d\mu\bigg|\leq C \|v\|_{BV(\Omega)}.
\end{equation}

Then,
\begin{eqnarray*}
\bigg|\int_{\R^n} u d\mu\bigg|=\bigg|\int_\Omega u d\mu\bigg|&=&\bigg|\int_\Omega v d\mu\bigg|\leq C\, \|v\|_{BV(\Omega)}, \tn{ by } \eqref{omegita1} \\
&=& C\, \norm{v}_{L^1(\O)} + C\, \int_{\O}|Dv| \\
&\leq & C\, \norm{v}_{L^1(\O)} + C\, \int_{\R^n}|Du|  , \tn{ by  } \eqref{omegita2} \\
&\leq& C\norm{v}_{L^{\frac{n}{n-1}}(\O)} + C\int_{\R^n}|Du|, \tn{ since } \O \tn{ is bounded} \\
&=& C\norm{u}_{L^{\frac{n}{n-1}}(\R^n)} + C\int_{\R^n}|Du|\\
&\leq&  C\int_{\R^n}|Du|= \norm{u}_{BV(\R^n)}, \tn{ by Theorem }\ref{mesorprendi},
\end{eqnarray*}
which implies that $\mu\in BV(\Bbb R^n)^*$.
Thus, Theorem  \ref{superestrella}  gives 
\begin{equation*}
|\mu (U)|\leq C \H^{n-1} (\partial U)
\end{equation*}
for every open set $U\subset \Bbb R^n$ with smooth boundary.

Conversely, assume that $\mu$ satisfies condition \eqref{condicion}. Then Theorem \ref{superestrella}  yields that $\mu\in BV(\Bbb R^n)^*$.
Let $u \in BV^\infty(\Omega)$ and consider its extension $u_0\in BV(\Bbb R^n)$ as in Lemma \ref{ellema}.
Then, since $u_0\in BV_c^\infty(\R^n)$, there exists a constant $C$ such that
\begin{equation} \label{elcinco}
\left | \int_{\Bbb R^n} (u_0)^* d\mu \right | \leq C \|u_0\|_{BV(\Bbb R^n)}.
\end{equation}

Now, Lemma \ref{ellema} yields $\|u_0\|_{BV(\Bbb R^n)}\leq C\|u\|_{BV(\Omega)}$ and since $u_0\equiv 0$ on $\Bbb R^n\backslash\Omega$ and $u_0\equiv u$ on $\Omega$, we obtain from \eqref{elcinco} the inequality
\begin{equation}
\left | \int_\Omega u^* d\mu \right | \leq C \|u\|_{BV(\Omega)},
\end{equation}
which means that $\mu\in BV(\Omega)^*$.

\end{proof}

\begin{remark} It is easy to see that if $\mu$ is a {\it positive} measure in $BV(\O)^*$ then its action on $BV(\O)$ is given by 
$$<\mu, u>=\int_{\O} u^{*} d\mu$$ 
for all $u\in BV(\O)$.
\end{remark}

\end{document}